\documentclass[a4paper,11pt]{article}
\usepackage[a4paper,total={6.5in, 8in}]{geometry}

\usepackage[backend=bibtex,style=numeric-comp,maxbibnames=99,bibencoding=utf8,giveninits=true]{biblatex}
\bibliography{refs_xvem}

\usepackage{comment}
\usepackage{amsmath}       
\usepackage{amssymb}       
\usepackage{color}
\usepackage{bm}
\usepackage{latexsym}
\usepackage{times}
\usepackage{epsfig}
\usepackage{graphicx,graphics,rotating}
\usepackage{multicol}
\usepackage{multirow}
\usepackage{cancel}
\usepackage{enumitem}
\usepackage[colorlinks,allcolors={blue}]{hyperref}
\usepackage{authblk}

\usepackage{caption}
\usepackage{nicefrac}

\usepackage{pgfplots,pgfplotstable}
\pgfplotsset{compat=1.16}
\usepackage{booktabs}
\usepackage{url}
\usepackage{float}
\usepackage{subcaption}

\newtheorem{theorem}{Theorem}[section]
\newtheorem{lemma}[theorem]{Lemma}

\newtheorem{remark}[theorem]{Remark}

\newtheorem{assumption}[theorem]{Assumption}

\definecolor{MyDarkGreen}{rgb}{0,0.45,0}

\def\trait #1 #2 #3 {\vrule width #1pt height #2pt depth #3pt}
\def\fin{\hfill
        \trait .3 5 0
        \trait 5 .3 0
        \kern-5pt
        \trait 5 5 -4.7
        \trait 0.3 5 0
\\\smallskip}

\newenvironment{proof}[1][]{\textit{Proof#1.}~}{\fin}

\newcommand{\XVEM}{\text{X-VEM}}

\newcommand{\REAL}{\mathbb{R}}

\newcommand{\xv}{\bm{x}}
\newcommand{\yv}{\bm{y}}

\newcommand{\as}{a}

\newcommand{\fs}{f}

\newcommand{\hs}{h}

\newcommand{\qs}{q}

\newcommand{\us}{u}
\newcommand{\vs}{v}
\newcommand{\ws}{w}

\newcommand{\zs}{z}

\newcommand{\calD}{\mathcal{D}}
\newcommand{\calE}{\mathcal{E}}

\newcommand{\calI}{\mathcal{I}}

\newcommand{\calV}{\mathcal{V}}

\newcommand{\PS} [1]{\mathbb{P}_{#1}}

\newcommand{\HONE}  {H^1}

\newcommand{\LTWO}  {L^2}

\newcommand{\HONEzr}{H^1_0}

\newcommand{\LS}[1] {L^{#1}}
\newcommand{\HS}[1] {H^{#1}}

\newcommand{\CS}[1] {C^{#1}}

\newcommand{\TEXTFONTA}[1]{#1} 
\renewcommand{\P} {\TEXTFONTA{E}} 
\newcommand  {\E} {\TEXTFONTA{e}} 
\newcommand  {\V} {\TEXTFONTA{v}} 

\newcommand{\hh}{h}
\newcommand{\Th}{\Omega_{\hh}}

\newcommand{\xvP}{\xv_{\P}}          
\newcommand{\xvE}{\xv_{\E}}          
\newcommand{\xvV}{\xv_{\V}}          

\newcommand{\hP}{\hh_{\P}}

\newcommand{\hE}{\hh_{\E}}

\newcommand{\mP}{\ABS{\P}}

\newcommand{\Eset}{\calE_h}           
\newcommand{\Vset}{\calV_h}           

\newcommand{\dS}{\,ds}
\newcommand{\dx}{\,d\xv}

\newcommand{\nor}  {\bm{n}}

\newcommand{\norP} {\bm{n}_{\P}}

\newcommand{\fsh}{\fs_{\hh}}

\newcommand{\ush}{\us_{\hh}}

\newcommand{\vsh}{{\vs_{\hh}}}

\newcommand{\asP}{\as^{\P}}

\newcommand{\ash}{\as_{\hh}}

\newcommand{\SP} {S^{\P}}

\newcommand{\snorm}[2]{|#2|_{#1}}

\newcommand{\norm}[2]{\|#2\|_{#1}}

\newcommand{\abs}    [1]{|#1|}

\newcommand{\ABS}    [1]{\left|#1\right|}

\newcommand{\SQBRAC}[2][]{\left[#2\right]_{#1}}

\newcommand{\angles}  [2]{\langle#1\rangle_{#2}}

\newcommand{\Hhalfprod}  [2]{\angles{#1}{-\frac12,\frac12,#2}}

\newcommand{\Vspace}{V}

\newcommand{\Vh}[1] {\Vspace^{\hh}_{#1}}
\newcommand{\Vhk}   {\Vh{k}}

\newcommand{\PinP}[1]{\Pi^{\nabla,\P}_{#1}}
\newcommand{\PizP}[1]{\Pi^{0,\P}_{#1}}
\newcommand{\PizE}[1]{\Pi^{0,\E}_{#1}}

\newcommand{\PitEo}[1]{\Pi^{\Psio}_{#1,\E}}

\newcommand{\restrict}[2]{\mbox{$#1$}_{|{#2}}}

\newcommand{\EOD}{\end{document}}

\newcommand{\qst}{\qs}
\newcommand{\usR}{\us^{r}}
\newcommand{\vsR}{\vs^{r}}

\newcommand{\usht}{\us_{\hh}}
\newcommand{\vsht}{\vs_{\hh}}
\newcommand{\wsht}{\ws_{\hh}}

\newcommand{\asht}{\as_{\hh}}

\newcommand{\Vhkt}  {\Vspace^{\Psi}_{k,\hh}}

\newcommand{\Vhktzero}  {\Vspace^\Psi_{k,\hh,0}}

\newcommand{\PiDP}[1]{\Pi^{\Delta}_{#1,\P}}
\newcommand{\PinPt}[1]{\Pi^{\nabla,\Psi}_{#1,\P}}
\newcommand{\Psio}{{\mathfrak P}}
\newcommand{\PSt}[1]{\mathbb{P}^\Psi_{#1}}
\newcommand{\PSto}[1]{\mathbb{P}^{\Psio}_{#1}}
\newcommand{\PSD}[1]{\mathbb{P}^{\Delta}_{#1}}

\newcommand{\DOFS}[1]{($\mathbf{#1}$)}
\newcommand{\Ihk}{\widehat{\calI}_{k,\hh}}
\newcommand{\Ihkt}{\calI_{k,\hh}}

\newcommand{\RESD}{\calE}
\newcommand{\discretenorm}[1]{\norm{1,\hh}{#1}}
\newcommand{\localseminorm}[1]{\norm{1,\P}{#1}}

\newcommand{\reverseLogLogSlopeTriangle}[5]
{
	\pgfplotsextra
	{
		\pgfkeysgetvalue{/pgfplots/xmin}{\xmin}
		\pgfkeysgetvalue{/pgfplots/xmax}{\xmax}
		\pgfkeysgetvalue{/pgfplots/ymin}{\ymin}
		\pgfkeysgetvalue{/pgfplots/ymax}{\ymax}
		
		\pgfmathsetmacro{\xArel}{#1}
		\pgfmathsetmacro{\yArel}{#3}
		\pgfmathsetmacro{\xBrel}{#1-#2}
		\pgfmathsetmacro{\yBrel}{\yArel}
		\pgfmathsetmacro{\xCrel}{\xBrel}
		
		\pgfmathsetmacro{\lnxB}{\xmin*(1-(#1-#2))+\xmax*(#1-#2)} 
		\pgfmathsetmacro{\lnxA}{\xmin*(1-#1)+\xmax*#1} 
		\pgfmathsetmacro{\lnyA}{\ymin*(1-#3)+\ymax*#3} 
		\pgfmathsetmacro{\lnyC}{\lnyA+#4*(\lnxA-\lnxB)}
		\pgfmathsetmacro{\yCrel}{\lnyC-\ymin)/(\ymax-\ymin)}
		
		\coordinate (A) at (rel axis cs:\xArel,\yArel);
		\coordinate (B) at (rel axis cs:\xBrel,\yBrel);
		\coordinate (C) at (rel axis cs:\xCrel,\yCrel);
		
		\draw[#5]   (A)-- node[pos=0.5,anchor=north] {\scriptsize{1}}
		(B)-- 
		(C) -- node[pos=0.,anchor=east] {\scriptsize{#4}}
		cycle;
	}
}

\newcommand{\email}[1]{\href{mailto:#1}{#1}}

\begin{document}

\title{The eXtended Virtual Element Method for elliptic problems with
  weakly singular solutions}

\author[1,2]{J\'er\^ome Droniou\footnote{\email{jerome.droniou@umontpellier.fr, jerome.droniou@monash.edu}}}
\affil[1]{IMAG, Univ. Montpellier, CNRS, Montpellier, France}
\affil[2]{School of Mathematics, Monash University, Clayton, Australia}
\author[3]{Gianmarco Manzini\footnote{\email{gmanzini@lanl.gov}}}
\affil[3]{T-5, Theoretical Division, Los Alamos National Laboratory, Los Alamos, NM, USA}
\author[2]{Liam Yemm\footnote{\email{liam.yemm@monash.edu }}}

\maketitle

  \begin{abstract}
    This paper introduces a novel eXtended virtual element method, an
    extension of the conforming virtual element method. The \XVEM{} is
    formulated by incorporating appropriate enrichment functions in
    the local spaces. The method is designed to handle highly generic
    enrichment functions, including singularities arising from
    fractured domains. By achieving consistency on the enrichment
    space, the method is proven to achieve arbitrary approximation
    orders even in the presence of singular solutions. The paper
    includes a complete convergence analysis under general assumptions on mesh
    regularity, and numerical experiments validating the method's
    accuracy on various mesh families, demonstrating optimal
    convergence rates in the $L^2$- and $H^1$-norms on fractured or L-shaped domains.
    
  \end{abstract}



\section{Introduction}

Introduced by \cite{BeiraodaVeiga-Brezzi-Cangiani-Manzini-Marini-Russo:2013, Ahmad-Alsaedi-Brezzi-Marini-Russo:2013}, the conforming virtual element method (VEM) is a modern numerical technique for the discretisation of partial differential equations. At its core, VEM transcends traditional finite element methods by allowing for a flexible and versatile discretisation strategy, wherein the computational domain is partitioned into polytopes (polygons in 2D, polyhedra in 3D) of arbitrary shapes. The VEM is one of several polytopal methods developed in the last couple of decades or so, which also include Discontinuous Galerkin methods \cite{Di-Pietro.Ern:12,Antonietti.Brezzi.ea:09,Cangiani.Dong.ea:16}, Hybrid Discontinuous Galerkin \cite{Cockburn.Gopalakrishnan.ea:09,Cockburn.Di-Pietro.ea:16}, Hybrid High-Order (HHO) \cite{hho-book,Di-Pietro.Ern.ea:14}, and Weak Galerkin \cite{Mu.Wang.ea:15}, among others. We also note that some of these methods drew inspiration and/or turned out to be arbitrary-order methods of older low-order schemes, in particular from the Mimetic Finite Differences or Finite Volume families \cite{Kuznetsov.Lipnikov.ea:04,Droniou.Eymard.ea:10,Eymard.Gallouet.ea:10,Beirao-da-Veiga.Lipnikov.ea:14,Droniou:14}. VEM’s distinctive feature lies in its use of a ‘virtual space’ for which basis functions are not known explicitly, avoiding the need for predefined shape functions. 

Errors estimates of arbitrary-order methods, including VEM, are typically limited by the regularity of the exact solution. On non-smooth domains (such as regions with non-convex corners or those possessing cracks) it is expected that the exact solution to elliptic problems will contain weak singularities ~\cite{Grisvard:1985}. Singular behaviour can also stem from the problem data, for example, the boundary conditions in the
Motz problem ~\cite{Motz:1947}. This lack of regularity is well documented in the finite element literature and is typically overcome through enriched approximations based on a partition of unity method \cite{melenk.babuska:1996:partition, babuska.melenk:1997:partition}. The extended finite element method \cite{belytschko.black:1999:elastic, moes.dolbow:1999:finite} is one such method, originally designed to handle discontinuities in crack growth models. In particular, by enriching the local spaces with basis functions that are discontinuous across the crack, the method allows for optimal approximation without the need for mesh refinement near the discontinuity. Finite element methods, whether standard or enriched, are restricted to specific meshes (mostly composed of triangles/tetrahedra or quadrangles/hexahedra), which limits their flexibility. For instance, preserving mesh regularity during refinement necessitates a large number of elements and layers, hindering local and adaptive refinements. Similarly, mesh agglomeration (e.g., for multigrid solvers) and mesh cutting (e.g., for domains with interfaces) result in admissible finite element meshes only in very specific, often trivial, cases. For these reasons, polytopal methods have received a lot of attention since the early 2000s.

To blend the benefits of polytopal methods (mesh flexibility) together with enriched finite elements (optimal convergence for solutions with singularities), a  trend in the literature recently arose around designing and analysing \emph{enriched polytopal methods}.
The work of \cite{artioli.mascotto:2021:enrichment} designs an enriched non-conforming virtual element method (NC-VEM) for the Poisson problem targeting singularities arising from the domain geometry (such as corners or cracks). However, the analysis there requires to assume a discrete trace inequality, which depends on the enrichment function. Following this approach, the same authors have defined an enriched NC-VEM for a plane elasticity problem with corner singularities \cite{artioli.mascotto:2023:enriched}. An extended HHO method has been designed for the Poisson problem \cite{yemm:2022:design} targeting general enrichments with (locally) square-integrable Laplacian and Neumann traces and avoids the use of discrete inequalities dependent on the enrichment function. We also make note of the recent work of \cite{yemm:2024:enriched} which designs an extended HHO scheme for the Stokes problem with generic enrichment functions, albeit with limited regularity disallowing singularities arising from non-convex corners. To date, only \cite{benvenuti.chiozzi.ea:2019:extended, benvenuti.chiozzi.ea:2022:extended} have attempted to design an enriched \emph{conforming} virtual element method. However, both these methods only consider the lowest order VEM and lack a complete error analysis, mostly because of an approach based on the continuous enriched virtual space and the lack of inverse inequalities for enriched functions or degrees of freedom. In this paper, an eXtended virtual element method (\XVEM{}) is designed which is valid for highly generic enrichment functions, offers arbitrary approximation orders and is accompanied with a rigorous and complete analysis; to overcome the shortcomings of previous (partial) analyses mentioned above, we adopt a ``fully discrete'' approach closer to the approach used, e.g., for HHO schemes. Moreover, and contrary to the previously mentioned enriched NC-VEM, the method is not specific to harmonic singularities or those arising from cracks or corners in the domain, but is valid for a generic enrichment space satisfying Assumption \ref{assum:regularity} below. This assumption is still valid for singularities arising from fractured domains -- something the extended HHO method fails to capture.

We perform the convergence analysis under suitable but still general assumptions on the mesh regularity, and we derive the estimates for the approximation error. We obtain error estimates that depend solely on the regular component of the continuous solution, both in a discrete $H^1$-like norm based on the degrees of freedom and in the (computable) broken $H^1$-norm using the enriched elliptic reconstruction.
To assess the method's behaviour and confirm the theoretical
expectations, we conduct a set of numerical experiments on several mesh families, including meshes with convex and non-convex elements on the $L$-shaped domain with a corner singularity.
Our results assess the accuracy of the extended virtual element method and demonstrate its optimal rates of convergence in the $\LS{2}$- and $\HS{1}$-norm. As is well-known for singularity-enriched schemes, to avoid ill-conditioning of the system a local enrichment is preferable, through which only the DOFs in a portion of the domain are enriched with the singular function; contrary to the case of HHO enriched methods \cite{yemm:2024:enriched, yemm:2022:design} our analysis does not cover local enrichment, which remains for a future work, but we numerically demonstrate that it does work for the scheme developed here.

While an extension of the method designed in this paper to three-dimensions may be feasible, significant investigation into the design of virtual spaces on element boundaries and appropriate reconstruction operators would be required.

The outline of the paper is as follows.
In Section~\ref{sec:XVEM:formulation}, we introduce the model problem and its extended virtual element approximation (X-VEM).
In Section~\ref{sec3:theory}, we present the convergence
analysis of X-VEM. The exposition is done in dimension 2 for simplicity, but can be generalised to dimension 3 in a straightforward manner. 
In Section~\ref{sec5:numerical} we investigate the method's
performance on suitable numerical experiments.
Finally, in Section~\ref{sec6:conclusions} we offer some concluding remarks.
\subsection{Notation and technicalities}
\label{subsec:notation}

Throughout this paper, we adopt the notation of Sobolev spaces of~\cite{Adams-Fournier:2003}.
Accordingly, we denote the space of square-integrable functions
defined on any open, bounded, connected domain $\calD\subset\REAL^2$
with boundary $\partial\calD$ by $\LTWO(\calD)$, and the Hilbert space
of functions in $\LTWO(\calD)$ with all partial derivatives up to a
positive integer $m$ also in $\LTWO(\calD)$ by $\HS{m}(\calD)$.
The norm of $\LTWO(\calD)$ is written $\norm{\calD}{{\cdot}}$, while the norm and semi-norm in $\HS{m}(\calD)$ are respectively denoted by $\norm{\HS{m}(\calD)}{{\cdot}}$ and $\snorm{\HS{m}(\calD)}{{\cdot}}$; the latter is the sum of $\LTWO$-norms of derivatives of order $m$.

Let $\Omega \subset \REAL^2$ be an open, bounded polygonal domain with boundary $\Gamma$.
The virtual element method is formulated on the mesh family $\big\{\Th\big\}_{h}$, where each mesh $\Th$ is a partition of the computational domain $\Omega$ into non-overlapping polygonal elements $\P$.
A polygonal element $\P$ is a compact subset of $\REAL^2$ with
boundary $\partial\P$, area $\mP$ centre of mass $\xvP$, and diameter $\hP=\sup_{\xv,\yv\in\P}\vert\xv-\yv\vert$.
The mesh elements $\Th$ form a finite cover of $\Omega$ such that $\overline{\Omega}=\cup_{\P\in\Th}\P$ and the mesh size labelling each mesh $\Th$ is defined by $\hh=\max_{\P\in\Th}\hP$.
A mesh edge $\E$ has centre $\xvE$ and length $\hE$; a mesh vertex $\V$ has position vector $\xvV$.
We denote the set of mesh edges by $\calE_{\hh}$ and the set of mesh vertices by $\calV_{\hh}$.
We denote the set of mesh edges by $\calE_{\hh}$ and the set of mesh vertices by $\calV_{\hh}$. 

The following regularity assumption is made on the mesh $\Th$. We note that this assumption allows for arbitrarily small (or numerous) edges.

\begin{assumption}[Regular mesh sequence]\label{assum:star.shaped}	
	There exists a constant \(\varrho>0\) such that every \(\P\in\Th\) is star-shaped with respect to a ball of radius $\varrho\hP$. 
\end{assumption}

For any integer number $\ell\geq0$, we let $\PS{\ell}(\P)$ and
$\PS{\ell}(\E)$ denote the space of polynomials defined on the element $\P$ and the edge $\E$, respectively; $\PS{\ell}(\Th)$ denotes the space of piecewise polynomials of degree $\ell$ on the mesh $\Th$.
%
  
  



\section{The eXtended virtual element method}
\label{sec:XVEM:formulation}

\subsection{Model problem}
\label{sec:model:problem}
We consider the Poisson problem with homogeneous Dirichlet boundary conditions
\begin{subequations}\label{eq:Poisson}
  \begin{align}
    -\Delta\us &= \fs \phantom{0} \text{in}\;\Omega,\label{eq:Poisson:a}\\
    \us        &= 0 \phantom{\fs} \text{on}\;\Gamma,\label{eq:Poisson:b}
  \end{align}
\end{subequations}
for the scalar unknown $\us$.
The extension to non-homogeneous Dirichlet boundary conditions or to different types of boundary conditions such as Neumann or Robin conditions is straightforward.

Let $\HONEzr(\Omega)$ be the subspace of the Sobolev space of
functions $\HONE(\Omega)$ with zero trace on $\Gamma$.
The variational formulation of problem~\eqref{eq:Poisson} reads as: Find $\us\in\HONEzr(\Omega)$ such that
\begin{align}
	\as(\us,\vs) = \int_\Omega \fs \vs \dx \qquad\forall\vs\in\HONEzr(\Omega), \label{eq:pblm:var}
\end{align}
where the bilinear form $\as(\cdot,\cdot)$ is given by
\begin{equation}
  \as(\us,\vs) = \int_\Omega\nabla\us\cdot\nabla\vs \dx.
  \label{eq:bilform}
\end{equation}
The continuity and coercivity of the bilinear form $\as(\cdot,\cdot)$ imply the existence and uniqueness of the solution according to the Lax-Milgram lemma.

\subsection{Enrichment space}

We assume that the exact solution of the variational problem~\eqref{eq:pblm:var} is the sum of two terms
$$
  \us=\usR+\psi,
$$
where $\usR$ is sufficiently smooth and $\psi$ is the (weakly) singular part.
We are interested in developing an extended virtual element method that incorporates $\psi$ in the design of the approximation space, in such a way that the approximation properties of the scheme only depend on the regular component $\usR$.
The singular component $\psi$ is assumed to be an element of a finite dimensional \emph{enrichment space} $\Psi(\Omega)$ satisfying the following regularity requirements.

\begin{assumption}[Regularity of the enrichment space]\label{assum:regularity}
    ${}$
    \begin{enumerate}
        \item $\Psi(\Omega) \subset \HONE(\Omega)\cap \CS{0}(\overline{\Omega})$. 
        \item $\Delta \Psi(\P) \subset \LTWO(\P)$ for all $\P\in\Th$.
    \end{enumerate}
\end{assumption}
Here and in the following, if $U\subset \overline{\Omega}$ we set $\Psi(U):=\{w|_U\,:\,w\in\Psi(\Omega)\}$.

\begin{remark}[Conforming enrichment space and local enrichment]\label{rem:conforming.psi}
	The requirement that the enrichment space is globally conforming is due to the fact that the analysis is based on splitting the \XVEM{} interpolant into a singular component and a regular VEM interpolant (see Theorem \ref{thm:energy.error} below). If the enrichment space is not contained in $H^1(\Omega)$, then by writing $u = \usR + \psi\in H^1(\Omega)$ it cannot hold that $\usR\in H^1(\Omega)$ and thus its regular VEM interpolant is not defined. This has the major drawback of not allowing for enrichment spaces defined piece-wise -- a typical method of local enrichment, see e.g. \cite{yemm:2022:design} in the context of HHO methods. However, numerical tests suggest that locally enriching in this manner yields valid results (see Section \ref{sec5:numerical}).
	
	An alternative, and perhaps more natural way to define the \XVEM{} interpolant, is as the element of the \XVEM{} space with the same degrees of freedom as the function being interpolated. This would no longer require splitting a function into its regular and singular components and interpolating them individually, thus allowing for the weaker assumption $\Psi(\P) \subset \HONE(\P)\cap \CS{0}(\P)$ for all $\P\in\Th$. However, proving optimal approximation properties of an interpolant defined in this manner seems difficult; see Remark \ref{rem:alternate.interpolator}.
\end{remark}

\subsection{\XVEM{} spaces and elliptic projector}

Consider an integer $k \ge 1$ and set $l = \max\{0, k - 2\}$. The following extended polynomial spaces are defined on each element $\P\in\Th$,
\begin{equation*}
     \PSt{k}(\P) := \PS{k}(\P) + \Psi(\P),\qquad \PSD{l}(\P) := \PS{l}(\P) + \Delta\Psi(\P),
\end{equation*}
and the following space is defined on each edge $\E\in\Eset$,
\begin{equation*}
     \PSt{k}(\E) := \PS{k}(\E) + \Psi(\E).
\end{equation*}                                
The local \XVEM{} space is defined as
\begin{align*}
  \Vhkt(\P) :=
  \big\{\,
  \vsh\in\HS{1}(\P)\,:\,
  &
  \Delta\vsh\in\PSD{l}(\P),\,\,
  \restrict{\vsh}{\partial\P}\in\CS{0}(\partial\P),\,\,
  \restrict{\vsh}{\E}\in\PSt{k}(\E)\,\,\forall\E\subset\partial\P\big\}.
 \label{eq:XVEM:space}
\end{align*}
We note that $\Vhkt(\P)$ contains both the regular local VEM space of \cite{BeiraodaVeiga-Brezzi-Cangiani-Manzini-Marini-Russo:2013} and the extended polynomial space $\PSt{k}(\P)$. For each $\E\in\Eset$, we denote by $\Psio_\E$ the $L^2$-orthogonal complement of $\PS{k}(\E)$ in $\PSt{k}(\E)$, so that $\PSt{k}(E)=\PS{k}(\E)\oplus\Psio_\E$, and we set $\PSto{k-2}(\E)=\PS{k-2}(\E)\oplus\Psio_\E$. Heuristically, $\Psio_\E$ represents the ``most singular'' component of $\PSt{k}(\E)$, obtained by only keeping the part of $\Psi(\E)$ that is orthogonal to $\PS{k}(\E)$ (that is, the part of $\Psi(\E)$ which is the furthest away from polynomials of degree $\le k$).

The $L^2$-orthogonal projectors onto each of the spaces $\PSD{l}(\P)$ and $\PSto{k-2}(\E)$ are denoted by $\PiDP{l}$ and $\PitEo{k-2}$, respectively.
We show in Lemmas \ref{lem:boundary.unisolvence} and \ref{lem:unisolvence} below that each $\vsht \in \Vhkt(\P)$ is uniquely characterised by the following degrees of freedom:
\begin{itemize}
    \item[]\DOFS{D1} the values of $\vsht$ at the vertices of $\P$;  
    \item[]\DOFS{D2} the $L^2$-orthogonal projection $\PitEo{k-2}\vsht$ of $\vsht$ onto the space $\PSto{k-2}(\E)$ for each edge $\E\subset\partial\P$;
    \item[]\DOFS{D3} the $L^2$-orthogonal projection $\PiDP{l}\vsht$ of $\vsht$ onto the space $\PSD{l}(\P)$.
\end{itemize}

\begin{lemma}[Unisolvence of boundary values]\label{lem:boundary.unisolvence}
	The degrees of freedom \DOFS{D1}--\DOFS{D2} are unisolvent for the trace space of $\Vhkt(\P)$ on $\partial \P$.
\end{lemma}

\begin{proof}
We have to show that, for any $\vsht\in\Vhkt(\P)$, the degrees of freedom \DOFS{D1}--\DOFS{D2} entirely determine any $\vsht|_{\partial \P}$, and that for any choice of the values \DOFS{D1}--\DOFS{D2}, we can find $w\in C^0(\partial\P)$ with these degrees of freedom such that $w|_\E\in\PSt{k}(\E)$ for each $\E\subset\partial \P$.

\medskip

Let $\vsht\in\Vhkt(\P)$. To show that its boundary value is entirely determined by \DOFS{D1}--\DOFS{D2}, we only have to show that, for each $\E\in\Eset$, $v_\E:=\vsht|_\E$ is uniquely determined by its values at the endpoints $\partial\E$ of $\E$ and by $\PitEo{k-2}v_\E$.
Since $v_\E\in\PSt{k}(\E)=\PS{k}(\E)\oplus \Psio_\E$ we can write $v_\E=q+z$ with $q\in\PS{k}(\E)$ and $z\in\Psio_\E$. We first note that, by orthogonality of $\PS{k}(\E)$ and $\Psio_\E$, $\PitEo{k-2}v_\E=\PizE{k-2}q+z$, where $\PizE{k-2}$ is the $L^2$-orthogonal projection on $\PS{k-2}(\E)$. Hence, letting $\Pi^{0,\Psio}_\E$ be the $L^2$-orthogonal projection on $\Psio_\E$, we have
\begin{equation}\label{eq:v.bdry.z}
\Pi^{0,\Psio}_\E(\PitEo{k-2}v_\E)=\Pi^{0,\Psio}_\E(\PizE{k-2}q+z)=z,
\end{equation}
where the conclusion follows from the orthogonality of $\PS{k-2}(\E)$ and $\Psio_\E$.

We then recall that $v_\E=q+z$ to write
\begin{equation}\label{eq:v.bdry.q.1}
\PizE{k-2}q=\PizE{k-2}(v_\E-z)=\PizE{k-2}(\PitEo{k-2}v_\E),
\end{equation}
where the second equality is obtained using $\PS{k-2}(\E)\subset\PSto{k-2}(\E)$ (which ensures that $\PizE{k-2}=\PizE{k-2}\circ\PitEo{k-2}$) and the orthogonality of $z$ and $\PS{k-2}(\E)$.
Moreover
\begin{equation}\label{eq:v.bdry.q.2}
q(a)=v(a)-z(a)\quad \forall a\in\partial\E.
\end{equation}
The relation \eqref{eq:v.bdry.z} shows that 
\begin{equation}\label{eq:fix.z}
  \text{$z$ is uniquely determined by $\PitEo{k-2}v_\E$.}
\end{equation}
Since $q\in\PS{k}(\E)$, it is uniquely determined by $(q(a))_{a\in\partial \E}$ and $\PizE{k-2}q$, that is, according to \eqref{eq:v.bdry.q.1}--\eqref{eq:v.bdry.q.2}, by $(v(a)-z(a))_{a\in\partial\E}$ and $\PitEo{k-2}v_\E$; by \eqref{eq:fix.z}, we infer that  \begin{equation}\label{eq:fix.q}
  \text{$q$ is uniquely determined by $(v(a))_{a\in\partial\E}$ and $\PitEo{k-2}v_\E$.}
\end{equation}
Combining \eqref{eq:fix.z} and \eqref{eq:fix.q} shows that $v_\E=q+z$ is uniquely determined by $(v(a))_{a\in\partial\E}$ and $\PitEo{k-2}v_\E$, which is what we needed to prove.

\medskip

Conversely, take $(m_a)_{a\in\Vset\subset\partial\P}$ and $(r_\E)_{\E\in\Eset\cap\partial\P}$ values for \DOFS{D1}--\DOFS{D2} (each $m_a$ belongs to $\REAL$, each $r_\E$ belongs to $\PSto{k-2}(\E)$). Define $w:\partial\P\to \REAL$ the following way; for each $\E\subset\partial\P$, set $z_\E=\Pi^{0,\Psio}_\E r_\E\in\Psio_\E$, define $q_\E\in\PS{k-2}(\E)$ such that $\PizE{k-2}q_\E=\PizE{k-2}r_\E$ and $q_\E(a)=m_a-z_\E(a)$ for each $a\in\partial \E$, and set $w|_\E=q_\E+z_\E$. Then it can easily be checked that $w\in C^0(\partial \P)$ since its value at each vertex $a\in\partial\P$ is $m_a$ (from either side of $a$), and the arguments developed above to establish \eqref{eq:v.bdry.z}--\eqref{eq:v.bdry.q.2} show that its degrees of freedom \DOFS{D1}--\DOFS{D2} match the chosen values.
\end{proof}

\begin{remark}[Choice of boundary degrees of freedom]
A seemingly more natural choice for \DOFS{D2} would be to consider the projection of $\vsht$ on $\PS{k-2}(\E)+\Psi(\E)$ for each $\E\subset \partial\P$. However, such a choice does not seem to ensure the unisolvence stated in Lemma \ref{lem:boundary.unisolvence}.
\end{remark}

\begin{remark}[Computation of the boundary degrees of freedom]\label{rem:bdrDOF}
A basis for the space $\Psio_\E$ can be generated by considering a canonical basis for $\PS{k}(\E) + \Psi(\E)$, orthonormalising, and retrieving the degrees of freedom corresponding to $\Psio_\E$. Consequently, a basis for the space $\PSto{k-2}(\E)$ follows. This procedure provides orthonormal bases of local spaces, which are known to positively impact the conditioning of the global algebraic system; see, e.g., the analysis in \cite{cond_hho} for the HHO method.
\end{remark}

\begin{lemma}[Unisolvence]\label{lem:unisolvence}
  For all elements $\P\in\Th$, the values provided by the continuous linear functionals  \DOFS{D1}, \DOFS{D2}, \DOFS{D3} are unisolvent in the virtual element space $\Vhkt(\P)$.
\end{lemma}
\begin{proof}  
Let $\Vhkt(\partial\P)$ be the trace space on $\partial\P$ of $\Vhkt(\P)$. The mapping
\begin{align*}
\Vhkt(\partial\P)\times \PSD{l}(\P)\mapsto{}&\Vhkt(\P)\\
(w,r)\rightarrow {}&\vsht\mbox{ such that } -\Delta\vsht=r\mbox{ in $\P$ and } \vsht|_{\partial\P}=w
\end{align*}
is an isomorphism, by definition of $\Vhkt(\P)$ and the well-posedness of the Poisson problem. Hence, the dimension of $\Vhkt(\P)$ is equal to the dimension of $\Vhkt(\partial\P)\times \PSD{l}(\P)$ which, by Lemma \ref{lem:boundary.unisolvence}, is identical to the number of degrees of freedom \DOFS{D1}--\DOFS{D3}. 

Therefore, it remains to prove that if all DOFs of $\vsht\in\Vhkt(\P)$ vanish, then $\vsht = 0$. By Lemma \ref{lem:boundary.unisolvence}, we already know that if \DOFS{D1} and \DOFS{D2} vanish, then $\vsht = 0$ on $\partial \P$. Therefore, an integration by parts reveals
  \[
    \int_{\P}\nabla \vsht \cdot \nabla \vsht \dx = -\int_{\P}\vsht \cdot \Delta \vsht \dx = -\int_{\P}\PiDP{l}\vsht \cdot \Delta \vsht \dx = 0,
  \]
  where the last equality follows from $\PiDP{l} \vsht=0$ as a result of \DOFS{D3} vanishing. Therefore, $\vsht = \textrm{const}$, and this constant must be zero due to $\vsht = 0$ on $\partial \P$.
\end{proof}

The \emph{extended elliptic projector} $\PinPt{k}:\Vhkt(\P)\to\PSt{k}(\P)$ projects the extended VEM space
$\Vhkt(\P)$ onto the extended polynomial space $\PSt{k}(\P)$.
The projection $\PinPt{k}\vsht$ is the solution to the variational problem:
\begin{align}
  \int_{\P}\nabla\PinPt{k}\vsht\cdot\nabla\qst\dx &= \int_{\P}\nabla\vsht\cdot\nabla\qst\dx\quad\qst\in\PSt{k}(\P),\label{eq:PinPt:A}\\[0.5em]
  \int_{\P}\PinPt{k}\vsht\dx &= \int_{\P}\vsht\dx.\label{eq:PinPt:B}
\end{align}
Clearly, $\PinPt{k}\qst=\qst$ for each $\qst\in\PSt{k}(\P)$.
For all mesh elements $\P$ and all virtual element functions $\vsht\in\Vhkt(\P)$, the extended projection $\PinPt{k}\vsht$ is computable from the DOFs as stated and proved in the following lemma.
\begin{lemma}[Computability of the elliptic projector]
  Let $\P$ be an element of $\Th$ and $\vsht\in\Vhkt(\P)$. Then, the extended elliptic projection
  $\PinPt{k}\vsht\in\PSt{k}(\P)$ is computable using only the degrees of freedom \DOFS{D1}--\DOFS{D3} of $\vsht$.
\end{lemma}
\begin{proof}
  It holds, by an integration by parts, that for all $\qst\in\PSt{k}(\P)$
  \begin{align}
    \int_{\P}\nabla\PinPt{k}\vsht\cdot\nabla\qst\dx
    ={}& -\int_{\P}\vsht\Delta\qst\dx + \Hhalfprod{\norP\cdot\nabla\qst,\vsht}{\partial\P} 
    \label{eq:PinPt:computability}
  \end{align}
  where $\Hhalfprod{\cdot,\cdot}{\partial\P}$ denotes the duality product between $\HS{-1/2}(\partial \P)$ and $\HS{1/2}(\partial \P)$.
  The moments of $\vsht$ against $\Delta\qst\in\PSD{l}(\P)$ are known from \DOFS{D3} and $\vsht$ is known entirely on the boundary $\partial\P$ from \DOFS{D1}--\DOFS{D2} (see Lemma \ref {lem:boundary.unisolvence}.
  Therefore, all the terms in the right-hand side of~\eqref{eq:PinPt:computability} are computable, and so is $\nabla\PinPt{k}\vsht$. Since $\PSD{l}(\P)$ contains the constant functions, the integral of $\vsht$ over $\P$ (and thus that of $\PinPt{k}\vsht$) is computable from \DOFS{D3}. The entire elliptic projector is therefore computable from the DOFs.
\end{proof}

The global \XVEM{} space is constructed by patching the local spaces:
\begin{align*}
	\Vhkt :=
	\Big\{\,
	\vsh\in\HS{1}(\Omega)\,:\,
	\restrict{\vsh}{\P}\in\Vhkt(\P)
	\quad\forall\P\in\Th
	\,\Big\}.
\end{align*}
Each $\vsht\in\Vhkt$ is uniquely defined by the following degrees of freedom:
\begin{itemize}
	\item[]\DOFS{D1} the values of $\vsht$ at each vertex of $\Vset$;
	\item[]\DOFS{D2} the $L^2$-orthogonal projection $\PitEo{k-2}\vsht$ of $\vsht$ onto the space $\PSto{k-2}(\E)$ for each edge $\E\in\Eset$;
	\item[]\DOFS{D3} the $L^2$-orthogonal projection $\PiDP{l}\vsht$ of $\vsht$ onto the space $\PSD{l}(\P)$ for each element $\P\in\Th$.
\end{itemize}

\subsection{Formulation of the scheme and main results}

At this point, the construction of the \XVEM{} is straightforward and follows the usual procedural steps
\cite{BeiraodaVeiga-Brezzi-Cangiani-Manzini-Marini-Russo:2013}.
We define the discrete bilinear form $\ash:\Vhkt\times\Vhkt\to\REAL$ by assembling all elemental contributions
\begin{equation}\label{eq:aht.def}
    \ash(\usht,\vsht) := \sum_{\P\in\Th}\SQBRAC{\asP(\PinPt{k}\usht,\PinPt{k}\vsht) + \SP(\usht,\vsht)},
\end{equation}
where
\begin{align*}
  \asP(\us,\vs) = \int_{\P}\nabla\us\cdot\nabla\vs\dx
\end{align*}
and the stabilisation term $\SP(\cdot,\cdot):\Vhkt(\P)\times\Vhkt(\P)\to\REAL$ can be any computable (from the DOFs), symmetric, positive semi-definite bilinear form satisfying:
\begin{itemize}
\item \emph{Coercivity and boundedness on}
  $\textrm{ker}(\PinPt{k})$. For all
  $\vsh\in\Vhkt(\P)\cap\textrm{ker}(\PinPt{k})$:
  \begin{equation}\label{eq:stab.norm.equivalence}
    \hP^{-2}\norm{\P}{\PiDP{l}\vsht}^2 + \hP^{-1}\norm{\partial\P}{\vsht}^2 \lesssim \SP(\vsht,\vsht) \lesssim \hP^{-2}\norm{\P}{\PiDP{l}\vsht}^2 + \hP^{-1}\norm{\partial\P}{\vsht}^2.
  \end{equation}
\item \emph{Consistency on $\PSt{k}(\P)$}. For all
  $\wsht\in\PSt{k}(\P)$ and $\vsh\in\Vhkt(\P)$
  \begin{equation}\label{eq:stab.PSt.consistency}
    \SP(\wsht,\vsht) = 0.
  \end{equation}
\end{itemize}
Here and throughout the rest of this paper, the notation $a \lesssim b$ is used if there exists a constant $C > 0$ independent of the quantities $a$, $b$ and the mesh size $h$ such that $a \le C b$. The hidden constant depends only on $\Omega$, $k$ and the mesh regularity $\varrho$ (see Assumption \ref{assum:star.shaped} below).

Defining the local seminorm $\localseminorm{{\cdot}}:\Vhkt(\P)\to\REAL$ via
\begin{equation}\label{eq:local.discrete.norm.def}
  \localseminorm{\vsht}^2 := \norm{\P}{\nabla\PinPt{k}\vsht}^2 + \hP^{-2}\norm{\P}{\PiDP{l}(\vsht - \PinPt{k}\vsht)}^2 + \hP^{-1}\norm{\partial\P}{\vsht - \PinPt{k}\vsht}^2,
\end{equation}
and the global seminorm $\discretenorm{{\cdot}}:\Vhkt\to\REAL$ via
\begin{equation}\label{eq:global.discrete.norm.def}
  \discretenorm{\vsht}^2 := \sum_{\P\in\Th}\localseminorm{\vsht}^2,
\end{equation}
we infer from the definition \eqref{eq:aht.def} of $\asht$ and
 \eqref{eq:stab.norm.equivalence} that
\begin{equation}\label{eq:ah.norm.equivalence}
    \discretenorm{\vsht}^2 \lesssim \asht(\vsht,\vsht) \lesssim \discretenorm{\vsht}^2.
\end{equation}
We prove in Section \ref{sec3:theory} that $\discretenorm{{\cdot}}$ (and, thus, $\asht(\cdot,\cdot)^\frac12$)
defines a norm on the homogeneous subspace
\begin{equation*}
  \Vhktzero := \{\vsht\in\Vhkt:\vsht|_{\partial\Omega} = 0\}.
\end{equation*}

\begin{remark}[Example of stabilisation and choice of local semi-norm]
An example of stabilisation bilinear form satisfying \eqref{eq:stab.norm.equivalence} and \eqref{eq:stab.PSt.consistency} is given by:
\begin{equation}\label{eq:stab.def}
  \begin{aligned}
    \SP(\usht,\vsht) ={}& \hP^{-2}\int_{\P}\left(\PiDP{l}(\usht - \PinPt{k}\usht)\right)\left( \PiDP{l}(\vsht - \PinPt{k}\vsht)\right)\dx \\ 
    &+ \hP^{-1}\int_{\partial\P}\left(\usht - \PinPt{k}\usht\right)\left(\vsht - \PinPt{k}\vsht\right)\dS.
  \end{aligned}
\end{equation}
Indeed, if $\vsh\in\textrm{ker}(\PinPt{k})$ then $\SP(\vsht,\vsht)=\hP^{-2}\norm{\P}{\PiDP{l}\vsht}^2 + \hP^{-1}\norm{\partial\P}{\vsht}^2$ so
\eqref{eq:stab.norm.equivalence} holds and, if $\wsht\in\PSt{k}(\P)$, then $\PinPt{k}\wsht=\wsht$ and $\SP(\wsht,\cdot)=0$, which proves \eqref{eq:stab.PSt.consistency}.

As mentioned in the introduction, to circumvent the issues of discrete inverse/trace inequalities on non-polynomial spaces (involving the singular function), we adopt an analysis that is inspired by those developed for fully discrete methods, such as the Hybrid High-Order or Discrete De Rham methods. For such methods, two (equivalent) norms are classically introduced in the analysis: one based on all the degrees of freedom (internal and boundary), and one based on the available (higher-order) local polynomial reconstruction plus terms penalising the difference between this reconstruction and the boundary degrees of freedom; see, e.g., \cite[Eqs. (2.7) and (2.15)]{hho-book} and \cite[Sections 4.4 and 4.5]{Di-Pietro.Droniou:23}. Here, we bypass this need for two norms and only consider the latter one: \eqref{eq:local.discrete.norm.def} is a discrete $H^1$-norm built from the gradient of the local reconstruction $\PinPt{k}\vsht$, plus penalisation terms between this reconstruction and the degrees of freedom $\PiDP{l}\vsht$ and $(\vsht)_{|\partial\P}$. The stabilisation \eqref{eq:stab.def} is then chosen to simply embed these penalisation terms, and as a consequence, for this particular choice, the inequalities in \eqref{eq:ah.norm.equivalence} are equalities.
\end{remark}

To approximate the volumetric term on the right-hand side of
\eqref{eq:pblm:var}, the forcing term $f$ is replaced with its
orthogonal projection $\fsh|_\P=\PiDP{l}\fs$ for all $\P\in\Th$, which shows that
\begin{equation}\label{eq:approx.rhs}
  \int_\Omega \fsh\vsht\dx := \sum_{\P\in\Th} \int_\P (\PiDP{l}\fs) \vsht\dx = \sum_{\P\in\Th} \int_\P \fs (\PiDP{l}\vsht)\dx
\end{equation}
is computable from the DOFs of $\vsht$ since $f$ is known.
The \XVEM{} scheme then reads: Find $\usht\in\Vhktzero$ such that
\begin{equation}\label{eq:discrete.problem}
  \asht(\usht, \vsht) =\int_\Omega \fsh\vsht \qquad \forall \vsht \in \Vhktzero.
\end{equation}

\begin{theorem}[Discrete Energy Error]\label{thm:energy.error}
  Let $\usht\in\Vhktzero$ denote the solution to the \XVEM{} scheme
  \eqref{eq:discrete.problem} and $\us=\usR + \psi \in
  \HONEzr(\Omega)$, with $\usR\in\HS{k+1}(\Th)$ and $\psi\in\Psi(\Omega)$, the solution to the continuous problem
  \eqref{eq:pblm:var}.
  Under Assumption \ref{assum:regularity}, the following energy error estimate holds:
  \begin{equation}\label{eq:energy.error}
    \discretenorm{\ush - \Ihkt \us} \lesssim \hh^k \snorm{\HS{k+1}(\Th)}{\usR},
  \end{equation}
  where $\Ihkt \us = \Ihk \usR + \psi \in \Vhktzero$ with $\Ihk:\HONE(\Omega)\to\Vhk$ the standard VEM interpolant of $\usR$, see \cite{Ahmad-Alsaedi-Brezzi-Marini-Russo:2013}.
\end{theorem}

We note in passing that splitting the interpolant $\Ihkt$ into a singular component and an element of the regular VEM space is crucial to the analysis and relies on the fact that the space $\Vhkt(\P)$ contains the regular VEM space.

\begin{remark}[Alternate interpolator and local enrichment]\label{rem:alternate.interpolator}
A more natural way to define the interpolant on $\Vhkt$ is to take $\Ihkt \vs$ as the unique virtual function that has $((\PiDP{l}\vs)_{\P\in\Th}, (\PitEo{k-2}\vs)_{\E\in\Eset}, (\vs(\nu))_{\nu\in\Vset})$ as degrees of freedom \DOFS{D1}--\DOFS{D3}. With this definition, $\Ihkt$ leaves the enrichment space $\Psi$ invariant, and allows for local enrichment as briefly discussed in Remark \ref{rem:conforming.psi}.

For the upcoming analysis, we would need to establish optimal approximation properties of $\PinPt{k}\Ihkt$, which essentially requires to prove the boundedness of this mapping (defined on $H^2(\P)$) in a scaled $H^2$-seminorm (see \cite[Lemma 1.43 and Section 5.5.6.2]{hho-book}). Proving this boundedness is however not a trivial matter for the extended VEM, as discrete inverse and trace inequalities are not readily available in the non-polynomial spaces $\PSt{k}(\P)$ and $\PSD{l}(\P)$ (we note that it is already quite challenging for the regular VEM \cite{Brenner-Guan-Sung:2017:SEV}). A more flexible approach to go in this direction would perhaps to use a fully discrete analysis (without direct usage of virtual functions, that are difficult to estimate), in the spirit of \cite{yemm:2022:design}.
\end{remark}

\begin{theorem}[$H^1$ Error]\label{thm:energy.error.continuous}
Let $\usht\in\Vhktzero$ and $\us\in\HONEzr(\Omega)$ be as in Theorem \ref{thm:energy.error}.
Under Assumption \ref{assum:regularity}, the following energy error estimate holds:
\begin{equation}\label{eq:energy.error.continuous}
	\Big(\sum_{\P\in\Th}\snorm{\HS{1}(\P)}{\PinPt{k}\ush - \us}^2\Big)^\frac12 \lesssim \hh^k \snorm{\HS{k+1}(\Th)}{\usR}.
\end{equation}
\end{theorem}



\section{Convergence analysis}
\label{sec3:theory}

We provide in this section a proof of Theorem \ref{thm:energy.error}. The proof hinges on the consistency of the scheme which we state in Theorem \ref{thm:consistency.error}. We begin with the following lemma which guarantees the stability and well-posedness of the scheme, and is also crucial in validating Theorem \ref{thm:energy.error}.
Given a mesh sequence satisfying Assumption \ref{assum:star.shaped}, the following continuous trace inequality holds \parencite{brenner.sung:2018:virtual}:
\begin{lemma}[Continuous trace inequality]
	For all $\P\in\Th$ and \(v\in H^1(\P)\), it holds,
	\begin{equation}\label{eq:continuous.trace}
		\hP\norm{\partial \P}{v}^2 \lesssim \norm{\P}{v}^2+\hP^2\norm{\P}{\nabla v}^2,
	\end{equation}
	where the hidden constant depends only on the mesh regularity parameter $\varrho$.
\end{lemma}

\begin{lemma}\label{lem:discrete.norm}
    The mapping $\discretenorm{{\cdot}}:\Vhktzero\to[0,\infty)$ defined by \eqref{eq:global.discrete.norm.def} is a norm.
\end{lemma}

\begin{proof}
    As $\discretenorm{ {\cdot} }$ is clearly a seminorm, we only have to prove that if $\discretenorm{\vsht} = 0$ for some $\vsht\in\Vhktzero$, then $\vsht=0$. To this end, we note that $\discretenorm{\vsht} = 0$ implies $\norm{\P}{\nabla \PinPt{k}\vsht}=0$ for each $\P\in\Th$, and thus that $\PinPt{k}\vsht$ is constant on each $\P\in\Th$. The condition $\norm{\partial\P}{\vsht-\PinPt{k}\vsht}=0$ shows that $\vsht|_{\partial\P}=(\PinPt{k}\vsht)|_{\partial\P}$ is also constant. Working from neighbour to neighbour and using the homogeneous condition $\vsht|_{\partial\Omega} = 0$, we infer that those constants are all zero. Combining with the condition
    \[
    \norm{\P}{\PiDP{l}(\vsht-\PinPt{k}\vsht)} = 0
    \]
    we infer that for all $\P\in\Th$, $\PiDP{l}\vsht = 0$. Therefore all the DOFs of $\vsht$ vanish, and thus $\vsht = 0$.
\end{proof}

\begin{lemma}[Consistency of $\SP$]\label{lem:stab.consistency}
    Let $\SP : \Vhkt(\P)\times\Vhkt(\P)\to\REAL$ be a stabilisation term satisfying equation \eqref{eq:stab.norm.equivalence}. Then it holds for all $\vs=\vsR + \psi \in \HONE(\P)$ with
    $\vsR\in\HS{k+1}(\P)$ and $\psi\in\Psi(\P)$ that
    \begin{equation}\label{eq:stab.consistency}
        \SP(\Ihkt\vs, \Ihkt\vs) \lesssim \left[\hP^k\snorm{\HS{k+1}(\P)}{\vsR}\right]^2.
    \end{equation}
\end{lemma}

\begin{proof}
    It follows from the definition of $\Ihkt$ and the $\PSt{k}(\P)$-consistency \eqref{eq:stab.PSt.consistency} of $\SP$ together with the fact that $\psi+ \PinPt{k}\Ihk\vsR\in\PSt{k}(\P)$ that
    \[
        \SP(\Ihkt\vs, \Ihkt\vs) = \SP(\Ihk\vsR+\psi, \Ihk\vsR+\psi)= \SP(\Ihk\vsR - \PinPt{k}\Ihk\vsR, \Ihk\vsR - \PinPt{k}\Ihk\vsR).
    \]
    Therefore, we infer from equation \eqref{eq:stab.norm.equivalence} that
    \begin{align}
        \SP(\Ihkt\vs, \Ihkt\vs) \lesssim{}& \hP^{-2}\norm{\P}{\PiDP{l}(\Ihk - \PinPt{k}\Ihk)\vsR}^2 + \hP^{-1}\norm{\partial\P}{(\Ihk - \PinPt{k}\Ihk)\vsR}^2.\label{eq:stab.consistency.proof.1}
    \end{align}
    By the $L^2$-boundedness of $\PiDP{l}$ and applying the continuous trace inequality \eqref{eq:continuous.trace} on the boundary term it follows that
    \begin{multline}\label{eq:stab.consistency.proof.2}
        \hP^{-2}\norm{\P}{\PiDP{l}(\Ihk - \PinPt{k}\Ihk)\vsR}^2 + \hP^{-1}\norm{\partial\P}{(\Ihk - \PinPt{k}\Ihk)\vsR}^2 \\ \lesssim \hP^{-2}\norm{\P}{(\Ihk - \PinPt{k}\Ihk)\vsR}^2 + \norm{\P}{\nabla(\Ihk - \PinPt{k}\Ihk)\vsR}^2.
    \end{multline}
    Substituting \eqref{eq:stab.consistency.proof.2} into \eqref{eq:stab.consistency.proof.1} and applying a Poincar\'{e}--Wirtinger inequality (due to the zero mean value of $(\Ihk - \PinPt{k}\Ihk)\vsR$ on $\P$) yields
    \[
         \SP(\Ihkt\vs, \Ihkt\vs)  \lesssim \norm{\P}{\nabla(\Ihk - \PinPt{k}\Ihk)\vsR}^2 \le \norm{\P}{\nabla(\Ihk - \PinP{k}\Ihk)\vsR}^2,
    \]
    where the substitution of $\PinPt{k}$ with $\PinP{k}$ (the elliptic projector on $\PS{k}(\P)$) is justified by the definition of $\PinPt{k}$, which ensures that $\nabla\PinPt{k}\Ihk\vsR$ is the best $L^2$-approximation of $\nabla\Ihk\vsR$ in $\nabla\PSt{k}(\P)\supset\nabla\PS{k}(\P)$, while $\nabla\PinP{k}\Ihk\vsR\in\nabla\PS{k}(\P)$. The proof of the consistency property \eqref{eq:stab.consistency} is complete by invoking the optimal approximation properties of $\PinP{k}\Ihk$ stated in \cite[Lemma 2.23]{Brenner-Guan-Sung:2017:SEV}.
\end{proof}

\begin{lemma}\label{lem:local.terms.boundedness}
    It holds for all $\vsht\in\Vhkt(\P)$ that
    \begin{equation}\label{eq:local.terms.boundedness}
        \norm{\P}{\nabla\PinPt{k}\vsht} + \hP^{-1}\norm{\P}{\PiDP{l}\vsht - \PinPt{k}\vsht} + \hP^{-\frac12}\norm{\partial\P}{\vsht - \PinPt{k}\vsht} + \SP(\vsht, \vsht)^\frac12 \lesssim \localseminorm{\vsht}.
    \end{equation}
\end{lemma}

\begin{proof}
    As each of $\norm{\P}{\nabla\PinPt{k}\vsht}$ and $\hP^{-\frac12}\norm{\partial\P}{\vsht - \PinPt{k}\vsht}$ appear in the definition of $\localseminorm{{\cdot}}$, their boundedness follows trivially. The boundedness of $\SP(\vsht, \vsht)^\frac12$ is a direct result of \eqref{eq:stab.norm.equivalence} applied to $\vsht-\PinPt{k}\vsht$ and of the consistency \eqref{eq:stab.PSt.consistency} which gives $\SP(\vsht-\PinPt{k}\vsht,\vsht-\PinPt{k}\vsht)=\SP(\vsht,\vsht)$. It remains only to show that
    \[
        \hP^{-1}\norm{\P}{\PiDP{l}\vsht - \PinPt{k}\vsht} \lesssim \localseminorm{\vsht}.
    \]
    To see this, we add and subtract the term $\PiDP{l}\PinPt{k}\vsht$ and apply a triangle inequality to yield
    \[
        \hP^{-1}\norm{\P}{\PiDP{l}\vsht - \PinPt{k}\vsht} \le \hP^{-1}\norm{\P}{\PiDP{l}\vsht - \PiDP{l}\PinPt{k}\vsht} + \hP^{-1}\norm{\P}{\PiDP{l}\PinPt{k}\vsht - \PinPt{k}\vsht}.
    \]
    The bound
    \[
        \hP^{-1}\norm{\P}{\PiDP{l}\vsht - \PiDP{l}\PinPt{k}\vsht} \le \localseminorm{\vsht}
    \]
    follows trivially from the definition of $\localseminorm{{\cdot}}$. As orthogonal projectors are the best approximations for their norm, and $\PS{l}(\P)\subset\PSD{l}(\P)$ it holds that
    \[
        \hP^{-1}\norm{\P}{\PiDP{l}\PinPt{k}\vsht - \PinPt{k}\vsht} \le \hP^{-1}\norm{\P}{\PizP{l}\PinPt{k}\vsht - \PinPt{k}\vsht},
    \]
    where $\PizP{l}$ is the $L^2(\P)$-orthogonal projector on $\PS{l}(\P)$. Applying the approximation properties of $\PizP{l}$ to yield
    \[
        \hP^{-1}\norm{\P}{\PizP{l}\PinPt{k}\vsht - \PinPt{k}\vsht} \lesssim \norm{\P}{\nabla\PinPt{k}\vsht}
    \]
    completes the proof.
\end{proof}

With $\us$ the exact solution to \eqref{eq:pblm:var}, which we assume to satisfy the conditions in Theorem \ref{thm:energy.error}, the consistency error is given by the linear form \(\RESD(\us;\cdot):\Vhktzero\to\REAL\) defined for all \(\vsht\in\Vhktzero\) as
 \begin{equation}\label{eq:RESD}
    \RESD(\us;\vsht):= \int_{\Omega}\fsh\vsht\dx - \ash(\Ihkt\us, \vsht).
\end{equation}

\begin{theorem}[Consistency error]\label{thm:consistency.error}
  Recalling that the exact solution to \eqref{eq:pblm:var} is written $\us=\usR+\psi$, the consistency error satisfies the estimate
    \begin{equation}\label{eq:RESD:error}
        \abs{\RESD(\us;\vsht)} \lesssim \hs^{k} \snorm{\HS{k+1}(\Th)}{\usR} \discretenorm{\vsht}.
    \end{equation}
\end{theorem}
\begin{proof}
    Applying the definition \eqref{eq:aht.def} of $\ash$ and the orthogonality properties of $\PinPt{k}$, it holds that
    \begin{equation}\label{eq:consistency.ash.Ihkt}
        \ash(\Ihkt \us, \vsht) = \sum_{\P\in\Th} \left(\int_{\P}\nabla\Ihkt\us\cdot\nabla\PinPt{k}\vsht\dx + \SP(\Ihkt \us, \vsht)\right).
    \end{equation}
    Consider now, on each element $\P\in\Th$, 
    \begin{equation}\label{eq:consistency.proof.1}
    \begin{aligned}
        -\int_\P \Delta \us \PiDP{l}\vsht\dx ={}& -\int_\P \Delta \us (\PiDP{l}\vsht - \PinPt{k}\vsht) \dx + \int_\P \nabla \us \cdot \nabla \PinPt{k}\vsht \dx \\
        &- \Hhalfprod{\nabla \us \cdot \nor,\PinPt{k}\vsht}{\partial\P},
    \end{aligned}
    \end{equation}
    which follows by adding and subtracting the term $\int_T \Delta \us \PinPt{k}\vsht \dx$ and integrating by parts. 
    By the continuity of the virtual function \(\vsht\) and of the normal fluxes $\nabla\us\cdot\nor$ across the mesh edges and \(\vsht|_{\partial\Omega}=0\), it holds that 
    \begin{equation}\label{eq:consistency.proof.bc}
        0
        = \Hhalfprod{\nabla \us \cdot \nor,\vsht}{\partial\Omega}
        = \sum_{\P\in\Th} \Hhalfprod{\nabla \us \cdot \norP,\vsht}{\partial\P}.
    \end{equation}
    Adding \eqref{eq:consistency.proof.bc} to \eqref{eq:consistency.proof.1} and summing over all elements $\P\in\Th$, it follows that
    \begin{multline}\label{eq:consistency.proof.2}
        -\sum_{\P\in\Th}\int_\P \Delta \us \PiDP{l}\vsht\dx = \sum_{\P\in\Th}\bigg[-\int_\P \Delta \us (\PiDP{l}\vsht - \PinPt{k}\vsht) \dx + \int_\P \nabla \us \cdot \nabla \PinPt{k}\vsht \dx  \\ + \Hhalfprod{\nabla \us \cdot \nor,\vsht - \PinPt{k}\vsht}{\partial\P}\bigg].
    \end{multline}
    Let $\zs \in \PSt{k}(\P)$. Since $\Delta \zs\in\Delta\PSt{k}(\P)\subset \PSD{l}(\P)$ and $\PiDP{l}$ is the $L^2$-projector on that space, we have
    \[
     \int_\P \Delta \zs (\PiDP{l}\vsht - \PinPt{k}\vsht) \dx = \int_\P \Delta \zs (\vsht - \PinPt{k}\vsht) \dx.
    \]
    Hence, integrating by part and using the definition \eqref{eq:PinPt:A} of $\PinPt{k}$,
    \[
        -\int_\P \Delta \zs (\PiDP{l}\vsht - \PinPt{k}\vsht) \dx + \Hhalfprod{\nabla \zs \cdot \nor,\vsht - \PinPt{k}\vsht}{\partial\P}= \int_\P \nabla \zs \cdot\nabla (\vsht - \PinPt{k}\vsht) \dx= 
        0.
    \]
    Applying this to $z=\psi+\PinP{k}\usR\in\PSt{k}(\P)$, we can cancel in the following expression the singular component $\psi$ of $u$ and introduce the term $\PinP{k}\usR$ to write
    \begin{multline*}
        -\int_\P \Delta \us (\PiDP{l}\vsht - \PinPt{k}\vsht) \dx + \Hhalfprod{\nabla \us \cdot \nor,\vsht - \PinPt{k}\vsht}{\partial\P} \\ = -\int_\P \Delta (\usR - \PinP{k}\usR) (\PiDP{l}\vsht - \PinPt{k}\vsht) \dx + \Hhalfprod{\nabla (\usR - \PinP{k}\usR) \cdot \nor,\vsht - \PinPt{k}\vsht}{\partial\P}.
    \end{multline*}
    Substituting back into \eqref{eq:consistency.proof.2} yields 
    \begin{multline}\label{eq:consistency.proof.3}
        -\sum_{\P\in\Th}\int_\P \Delta \us \PiDP{l}\vsht\dx = \sum_{\P\in\Th}\bigg[ \int_\P \nabla \us \cdot \nabla \PinPt{k}\vsht \dx \\ -\int_\P \Delta (\usR - \PinP{k}\usR) (\PiDP{l}\vsht - \PinPt{k}\vsht) \dx + \int_{\partial\P} \nabla (\usR - \PinP{k}\usR) \cdot \nor(\vsht - \PinPt{k}\vsht)\dS\bigg],
    \end{multline}
    where we have replaced the $\Hhalfprod{\cdot, \cdot}{\partial\P}$ duality product with the integral notation due to the regularity of $\usR$.
    Recalling the definition \eqref{eq:RESD} of $ \RESD(\us;\vsht)$ and invoking \eqref{eq:approx.rhs} together with $f=-\Delta u$, we infer
    \begin{align*}
        \RESD(\us;{}&\vsht) =
          -\sum_{\P\in\Th}\int_\P\Delta u\PiDP{l}\vsht\dx - \ash(\Ihkt\us, \vsht)
        \\
        ={}&\sum_{\P\in\Th}\bigg[ \int_\P \nabla (\us - \Ihkt\us) \cdot \nabla \PinPt{k}\vsht \dx - \SP(\Ihkt \us, \vsht) \\ 
        &-\int_\P \Delta (\usR - \PinP{k}\usR) (\PiDP{l}\vsht - \PinPt{k}\vsht) \dx + \int_{\partial\P} \nabla (\usR - \PinP{k}\usR) \cdot \nor(\vsht - \PinPt{k}\vsht)\dS\bigg],
    \end{align*}
    where the second equality follows from \eqref{eq:consistency.proof.3} and \eqref{eq:consistency.ash.Ihkt}.
    By applying Cauchy--Schwarz inequalities to each of the terms and \eqref{eq:local.terms.boundedness} we obtain
    \begin{align*}
     \RESD(\us;\vsht) \lesssim \sum_{\P\in\Th}\localseminorm{\vsht}\bigg[{}& \norm{\P}{\nabla (\us - \Ihkt\us)} + \SP(\Ihkt \us, \Ihkt \us)^\frac12 \\
     & + \hP\norm{\P}{\Delta (\usR - \PinP{k}\usR)} + \hP^\frac12\norm{\partial\P}{\nabla (\usR - \PinP{k}\usR) \cdot \nor} \bigg].
    \end{align*}
    It follows from the definition of $\Ihkt$ that $\us - \Ihkt\us = \usR - \Ihk\usR$ and thus applying the approximation properties of $\Ihk$ \cite[Lemma 2.23]{Brenner-Guan-Sung:2017:SEV} yields the bound
    \[
        \norm{\P}{\nabla (\us - \Ihkt\us)} \lesssim \hP^k\snorm{\HS{k+1}(\P)}{\usR}.
    \]
    Combining with the consistency \eqref{eq:stab.consistency} of $\SP$ and the approximation properties of $\PinP{k}$, it is clear that
    \[
        \RESD(\us;\vsht) \lesssim \sum_{\P\in\Th}\localseminorm{\vsht}\hP^k\snorm{\HS{k+1}(\P)}{\usR}.
    \]
    The result follows from a Cauchy--Schwarz inequality and the bound $\hP\le\hh$.
\end{proof}

The proof of Theorem \ref{thm:energy.error} now follows trivially.

\begin{proof}[ of Theorem \ref{thm:energy.error}]
    It follows from the coercivity property \eqref{eq:ah.norm.equivalence} and the Third Strang Lemma \cite[Lemma A.7]{hho-book} that
    \[
        \discretenorm{\ush - \Ihkt\us} \lesssim \sup_{\discretenorm{\vsht}=1} \abs{\RESD(\us;\vsht)}.
    \]
    The proof then follows from the residual error \eqref{eq:RESD:error}.
\end{proof}

\begin{proof}[ of Theorem \ref{thm:energy.error.continuous}]
	It follows from introducing $\PinPt{k}\us$ and $\PinPt{k}\Ihkt\us$ on each element $\P\in\Th$ and applying a triangle inequality that
	\begin{multline}\label{eq:energy.error.continuous.proof.1}
		\sum_{\P\in\Th}\snorm{\HS{1}(\P)}{\PinPt{k}\ush - \us}^2 \\
		\lesssim 
		\sum_{\P\in\Th}\snorm{\HS{1}(\P)}{\PinPt{k}(\ush - \Ihkt\us)}^2 + \sum_{\P\in\Th}\snorm{\HS{1}(\P)}{\PinPt{k}(\Ihkt\us - \us)}^2 + \sum_{\P\in\Th}\snorm{\HS{1}(\P)}{\PinPt{k}\us - \us}^2.
	\end{multline}
	By definition of $\Ihkt$, we have $\Ihkt\us - \us=\Ihk\usR - \usR$	and thus a bound on the second term in the right-hand side of \eqref{eq:energy.error.continuous.proof.1} follows from the $H^1(\P)$-boundedness of $\PinPt{k}$ and the approximation properties \cite[Lemma 2.23]{Brenner-Guan-Sung:2017:SEV} of $\Ihk$. After noticing that $\PinPt{k}\us - \us=\PinPt{k}\usR - \usR$ since $\PinPt{k}\psi=\psi$, the last term is bounded due to the optimal approximation properties of $\PinPt{k}$ in the $H^1(\P)$-seminorm. To bound the first term, note that
	\[
		\sum_{\P\in\Th}\snorm{\HS{1}(\P)}{\PinPt{k}(\ush - \Ihkt\us)}^2 \lesssim \discretenorm{\ush - \Ihkt\us}^2
	\]
	and the result follows from Theorem \ref{thm:energy.error}.
\end{proof}



\section{Numerical experiments}
\label{sec5:numerical}

    
  
  


The XVEM method described in Section \ref{sec:XVEM:formulation} is implemented using the open-source C++ library \texttt{PolyMesh} (\url{https://github.com/liamyemm/PolyMesh}), using linear algebra tools from the \texttt{Eigen3} library (see \url{http://eigen.tuxfamily.org}). We focus here on testing the method on domains possessing fractures or re-entrant corners. For both these cases, the exact solution is not expected to be $\HS{2}$, and so a classical virtual element method is expected to converge sub-optimally for all $k$.

Enriching elements and edges with singular functions far from the location of the singularity can cause severe ill-conditioning due to these functions being well-approximated by polynomials. This is typically mitigated through local enrichment \cite{artioli.mascotto:2021:enrichment, yemm:2022:design}. However, the analysis of the method described in this paper requires splitting the interpolant into a singular part, and an classical VEM interpolant on the regular part (see Remarks \ref{rem:conforming.psi} and \ref{rem:alternate.interpolator}). This requires the enrichment function to be globally $H^1$-conforming. As such, a sufficiently smooth cut-off function would be required to facilitate local enrichment. 

However, numerical tests suggest that removing the additional DOFs far from the singularity still leads to a valid scheme. In particular, we can enrich the local spaces on the elements and their edges only if the element intersects the disk of radius $\gamma>0$ centred at the singularity. For all tests the local enrichment parameter is taken as $\gamma = 0.15$. While such an enrichment is not accounted for by the analysis in this paper, the tests show that locally enriching in this manner achieves optimal results.

\subsection{Fractured domain}

Consider the fractured domain $\Omega=(-1, 1)^2\backslash\{(x,y):y=0, x>0\}$. The solution to problem \eqref{eq:Poisson} posed in this domain is expected to contain a singularity at the fracture tip of the form \cite{Grisvard:1985}
\[
    \psi(r, \theta) = r^\frac12 \sin(\frac12 \theta).
\]
As such, we test with an exact solution
\[
    u = \sin(\pi x) \sin (\pi y) + \psi
\]
and enrich the local spaces with the function $\psi$. 

A sequence of Cartesian meshes of the domain $\Omega$ is considered. The parameters of these meshes are displayed in Table \ref{table:cartesian.data} and two meshes are plotted in Figure \ref{fig:meshes} showing the fracture and the local enrichment scheme.

\begin{table}[!ht]
	\centering
	\pgfplotstableread{cartesian_mesh_data.dat}\loadedtable
	\pgfplotstabletypeset
	[
	columns={MeshTitle, MeshSize,NbCells,NbEdges,NbVertices}, 
	columns/MeshTitle/.style={column name=Mesh \#},
	columns/MeshSize/.style={column name=\(h\),/pgf/number format/.cd,fixed,zerofill,precision=4},
	columns/NbCells/.style={column name=Nb. Elements},
	columns/NbEdges/.style={column name=Nb. Edges},
	columns/NbVertices/.style={column name=Nb. Vertices},
	every head row/.style={before row=\toprule,after row=\midrule},
	every last row/.style={after row=\bottomrule} 
	]\loadedtable
	\caption{Mesh data used for the fractured domain test}
	\label{table:cartesian.data}
\end{table}

\begin{figure}[!ht]
	\centering
	\includegraphics[width=0.4\textwidth]{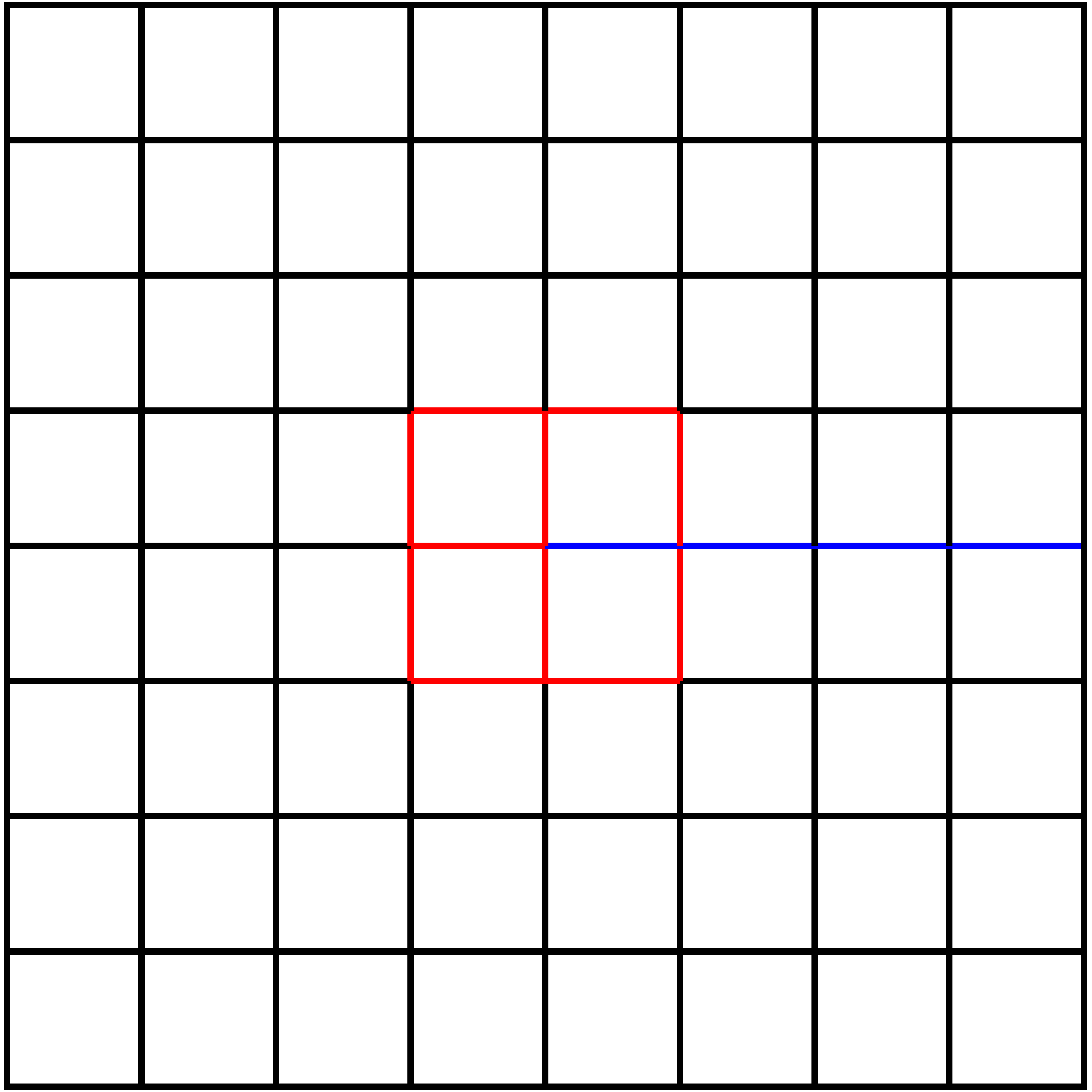}\hspace{0.5cm}
	\includegraphics[width=0.4\textwidth]{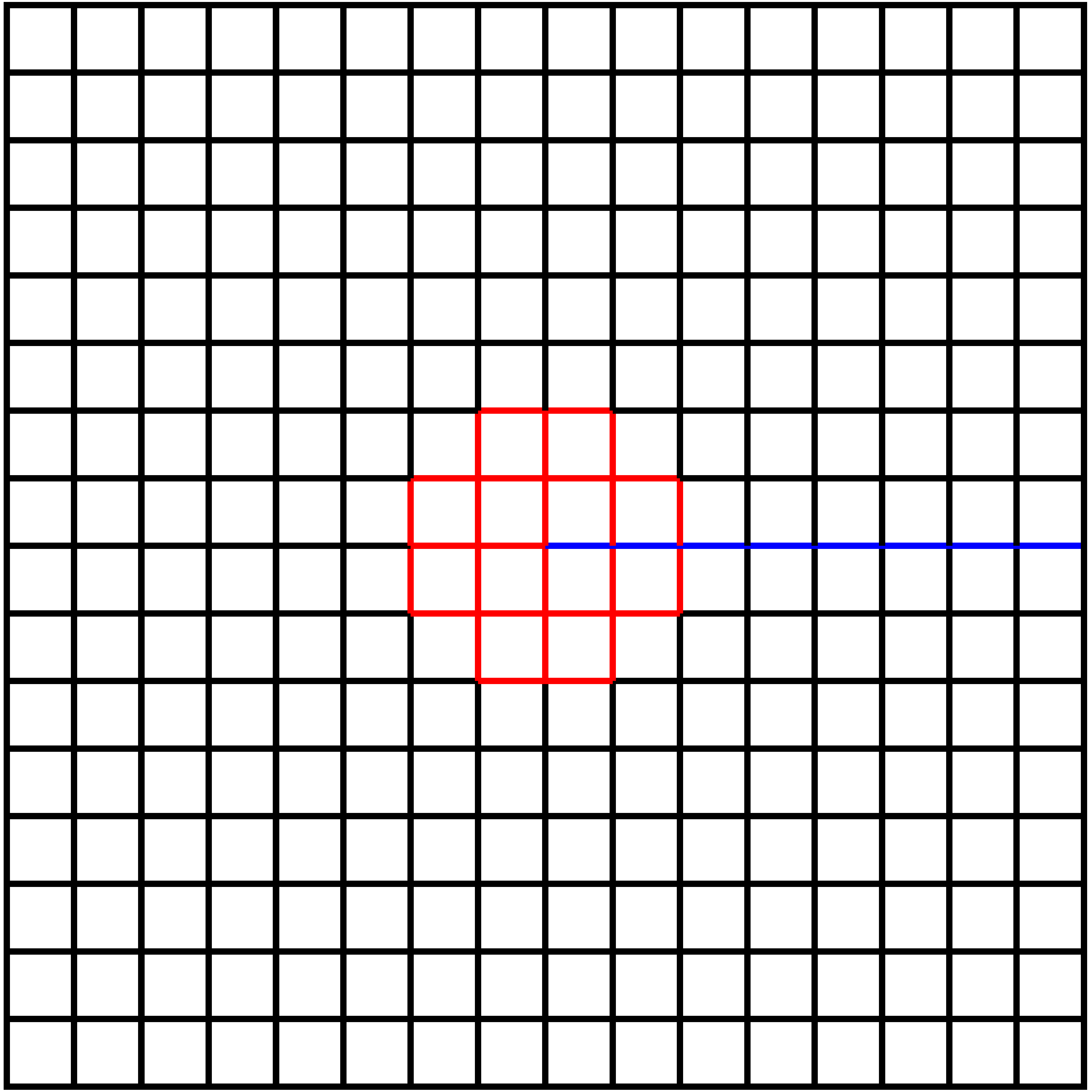} 
	\caption{Plots of meshes 2 and 3 used in the fractured domain test. The fracture is coloured blue and enriched elements are coloured red.}
	\label{fig:meshes}
\end{figure}

The solution to the discrete problem \eqref{eq:discrete.problem} is denoted by $\usht$ and the solution to the continuous problem \eqref{eq:Poisson} is denoted by $\us$. The accuracy of the scheme is determined by the following relative errors (respectively measuring the error in an $L^2$-like norm and an $H^1$-like seminorm):
\[
    E_{0,h}^2 := \frac{\sum_{\P\in\Th}\norm{\P}{\PinPt{k}(\usht - \us)}^2}{\sum_{\P\in\Th}\norm{\P}{\PinPt{k}\us}^2}\quad\textrm{and}\quad E_{1,h}^2 := \frac{\sum_{\P\in\Th}\snorm{\HS{1}(\P)}{\PinPt{k}(\usht - \us)}^2}{\sum_{\P\in\Th}\snorm{\HS{1}(\P)}{\PinPt{k}\us}^2}.
\]
By Theorem \ref{thm:energy.error.continuous}, for the \XVEM{} the error $E_{1,h}$ is expected to decay as $\mathcal O(h^k)$ despite the presence of the singularity in $\us$; we have not studied the convergence in $L^2$-norm but, classically, we would expect $E_{0,h}$ to decay as $\mathcal O(h^{k+1})$. On the contrary, for the non-enriched VEM these errors will be strongly limited by the lack of regularity of $\us$.

As enriching the virtual element method increases the number of degrees of freedom, it is most appropriate to plot the error against the degrees of freedom. After static condensation is performed, and the boundary degrees of freedom are fixed, the only remaining degrees of freedom are those corresponding to the internal edges and vertices. As the number of DOFs grows linearly with the number of elements, which is of order $\sim h^{-2}$ for our meshes, the $\mathcal O(h^k)$ error estimate predicted for the error in ($H^1$-like) energy norm by \eqref{eq:energy.error} translates to $\mathcal O((\sharp \text{DOFs})^{-k/2})$ in our graphs.

In Figure \ref{fig:fracture}, the error of the scheme is plotted against these remaining DOFs for a non-enriched, globally enriched and locally enriched scheme. The ill-conditioning of the globally enriched method is apparent due to the scheme failing on Meshes 5 \& 6 when $k=2,3$. Moreover, for a given mesh and polynomial degree $k$, the globally enriched scheme typically has a worse $H^1$ error than the locally enriched scheme, suggesting that ill-conditioning is plaguing the scheme. When $k=1, 2$, the non-enriched scheme converges optimally in $L^2$ error, however the convergence rate is sub-optimal in $H^1$ error. For $k=2$ the non-enriched scheme converges sub-optimally in both $L^2$ and $H^1$ error and is significantly outperformed by both enriched schemes. The locally enriched scheme converges optimally in all cases and is clearly the best performing scheme.

\begin{remark}[Rate of convergence in $L^2$ norm]
For $k=1,3$ we note an optimal convergence rate in $L^2$-norm, with an error decaying as $\mathcal O((\sharp\text{DOFs})^{-(k+1)/2})$. The rate of convergence in $L^2$-norm for $k=2$ however appears to be sub-optimal, and not better than the rate for the $H^1$-norm. Given our choice of (depleted) element unknowns in the VEM space, the VEM scheme for $k=2$ can be compared with the (depleted) HHO scheme for $(k,\ell)=(1,0)$ \cite[Chapter 5.1]{hho-book}, for which the sub-optimal convergence in $L^2$-norm is a well documented phenomenon (see \cite[Remark 5.17 and Section 5.1.8]{hho-book}).
\end{remark}

\begin{remark}[Managing ill-conditioning]
As discussed above, local enrichment leads to algebraic systems (both in computations of local operators, and for the final fully coupled system) that are much better conditioned than those obtained via global enrichment.

Another avenue to mitigate the ill-conditioning of the system obtained via global enrichment would be through the use of a pre-conditioner. Our current implementation solves the global linear system through a direct application of the BiCGStab solver from the Eigen3 library; it is expected that classical pre-conditioners, such as the BDDC method (see \cite{BDDC-VEM} for the application of this method to standard VEM), could lead to better numerical outcomes for the globally enriched scheme.
\end{remark}

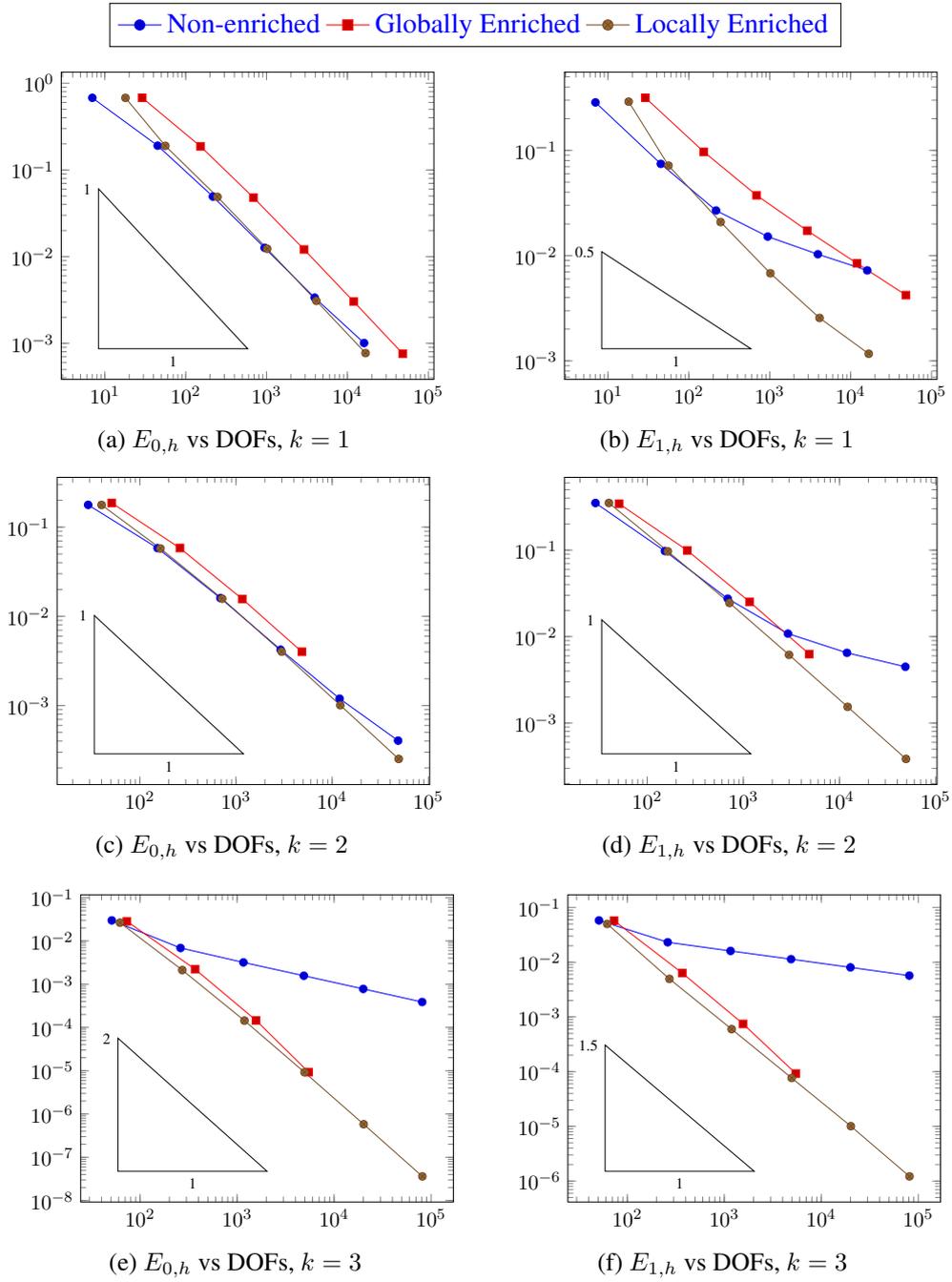
\begin{figure}[!ht]
	\centering
	\ref{legend_name}
	\vspace{0.25cm}\\
	\subcaptionbox{$E_{0,h}$ vs DOFs, $k=1$}
	{
		\begin{tikzpicture}[scale=0.75]
			\begin{loglogaxis}[ legend columns=3, legend to name=legend_name ]
				\legend{Non-enriched, Globally Enriched, Locally Enriched};
				\addplot table[x=DOFs,y=L2Error] {orthonormalised_fracture_nonenrich_k1.dat};
				\addplot table[x=DOFs,y=L2Error] {orthonormalised_fracture_enrich_k1.dat};
				\addplot table[x=DOFs,y=L2Error] {orthonormalised_fracture_locally_enrich_015_k1.dat};
				\reverseLogLogSlopeTriangle{0.50}{0.4}{0.1}{1}{black};
			\end{loglogaxis}
		\end{tikzpicture}
	}
	\hspace{0.5cm}\subcaptionbox{$E_{1,h}$ vs DOFs, $k=1$}
	{
		\begin{tikzpicture}[scale=0.75]
			\begin{loglogaxis}
				\addplot table[x=DOFs,y=H1Error] {orthonormalised_fracture_nonenrich_k1.dat};
				\addplot table[x=DOFs,y=H1Error] {orthonormalised_fracture_enrich_k1.dat};
				\addplot table[x=DOFs,y=H1Error] {orthonormalised_fracture_locally_enrich_015_k1.dat};
				\reverseLogLogSlopeTriangle{0.50}{0.4}{0.1}{0.5}{black};
			\end{loglogaxis}
		\end{tikzpicture}
	}
	\vspace{0.25cm}\\
	\subcaptionbox{$E_{0,h}$ vs DOFs, $k=2$}
	{
		\begin{tikzpicture}[scale=0.75]
			\begin{loglogaxis}
				\addplot table[x=DOFs,y=L2Error] {orthonormalised_fracture_nonenrich_k2.dat};
				\addplot table[x=DOFs,y=L2Error] {orthonormalised_fracture_enrich_k2.dat};
				\addplot table[x=DOFs,y=L2Error] {orthonormalised_fracture_locally_enrich_015_k2.dat};
				\reverseLogLogSlopeTriangle{0.50}{0.4}{0.1}{1}{black};
			\end{loglogaxis}
		\end{tikzpicture}
	}
	\hspace{0.5cm}\subcaptionbox{$E_{1,h}$ vs DOFs, $k=2$}
	{
		\begin{tikzpicture}[scale=0.75]
			\begin{loglogaxis}
				\addplot table[x=DOFs,y=H1Error] {orthonormalised_fracture_nonenrich_k2.dat};
				\addplot table[x=DOFs,y=H1Error] {orthonormalised_fracture_enrich_k2.dat};
				\addplot table[x=DOFs,y=H1Error] {orthonormalised_fracture_locally_enrich_015_k2.dat};
				\reverseLogLogSlopeTriangle{0.50}{0.4}{0.1}{1}{black};
			\end{loglogaxis}
		\end{tikzpicture}
	}
	\vspace{0.25cm}\\
	\subcaptionbox{$E_{0,h}$ vs DOFs, $k=3$}
	{
		\begin{tikzpicture}[scale=0.75]
			\begin{loglogaxis}
				\addplot table[x=DOFs,y=L2Error] {orthonormalised_fracture_nonenrich_k3.dat};
				\addplot table[x=DOFs,y=L2Error] {orthonormalised_fracture_enrich_k3.dat};
				\addplot table[x=DOFs,y=L2Error] {orthonormalised_fracture_locally_enrich_015_k3.dat};
				\reverseLogLogSlopeTriangle{0.50}{0.4}{0.1}{2}{black};
			\end{loglogaxis}
		\end{tikzpicture}
	}
	\hspace{0.5cm}\subcaptionbox{$E_{1,h}$ vs DOFs, $k=3$}
	{
		\begin{tikzpicture}[scale=0.75]
			\begin{loglogaxis}
				\addplot table[x=DOFs,y=H1Error] {orthonormalised_fracture_nonenrich_k3.dat};
				\addplot table[x=DOFs,y=H1Error] {orthonormalised_fracture_enrich_k3.dat};
				\addplot table[x=DOFs,y=H1Error] {orthonormalised_fracture_locally_enrich_015_k3.dat};
				\reverseLogLogSlopeTriangle{0.50}{0.4}{0.1}{1.5}{black};
			\end{loglogaxis}
		\end{tikzpicture}
	}
	\caption{Tests on fractured domain}
	\label{fig:fracture}
\end{figure}

\subsection{L-shaped domain, hexagonal meshes}\label{sec:Lshape.1}

Consider the fractured domain $\Omega=(-1, 1)^2\backslash [0, 1)^2$. The solution in this domain has a singularity located at the re-entrant corner of the form
\[
    \psi(r, \theta) = r^\frac23 \sin(\frac23 (\theta - \frac{\pi}{2})).
\]
Analogous to the fractured domain test, we consider an exact solution of the form
\[
    u = \sin(\pi x) \sin (\pi y) + \psi
\]
and enrich the local spaces with the function $\psi$. 

A sequence of hexagonal meshes of the domain $\Omega$ is considered. The parameters of these meshes are displayed in Table \ref{table:hexa.data} and two meshes are plotted in Figure \ref{fig:hexa.meshes} showing the enrichment scheme considered for the locally enriched tests.

\begin{table}[!ht]
	\centering
	\pgfplotstableread{hexa_mesh_data.dat}\loadedtable
	\pgfplotstabletypeset
	[
	columns={MeshTitle, MeshSize,NbCells,NbEdges,NbVertices}, 
	columns/MeshTitle/.style={column name=Mesh \#},
	columns/MeshSize/.style={column name=\(h\),/pgf/number format/.cd,fixed,zerofill,precision=4},
	columns/NbCells/.style={column name=Nb. Elements},
	columns/NbEdges/.style={column name=Nb. Edges},
	columns/NbVertices/.style={column name=Nb. Vertices},
	every head row/.style={before row=\toprule,after row=\midrule},
	every last row/.style={after row=\bottomrule} 
	]\loadedtable
	\caption{Mesh data used for the L-shaped domain test}
	\label{table:hexa.data}
\end{table}

\begin{figure}[!ht]
	\centering
	\includegraphics[width=0.4\textwidth]{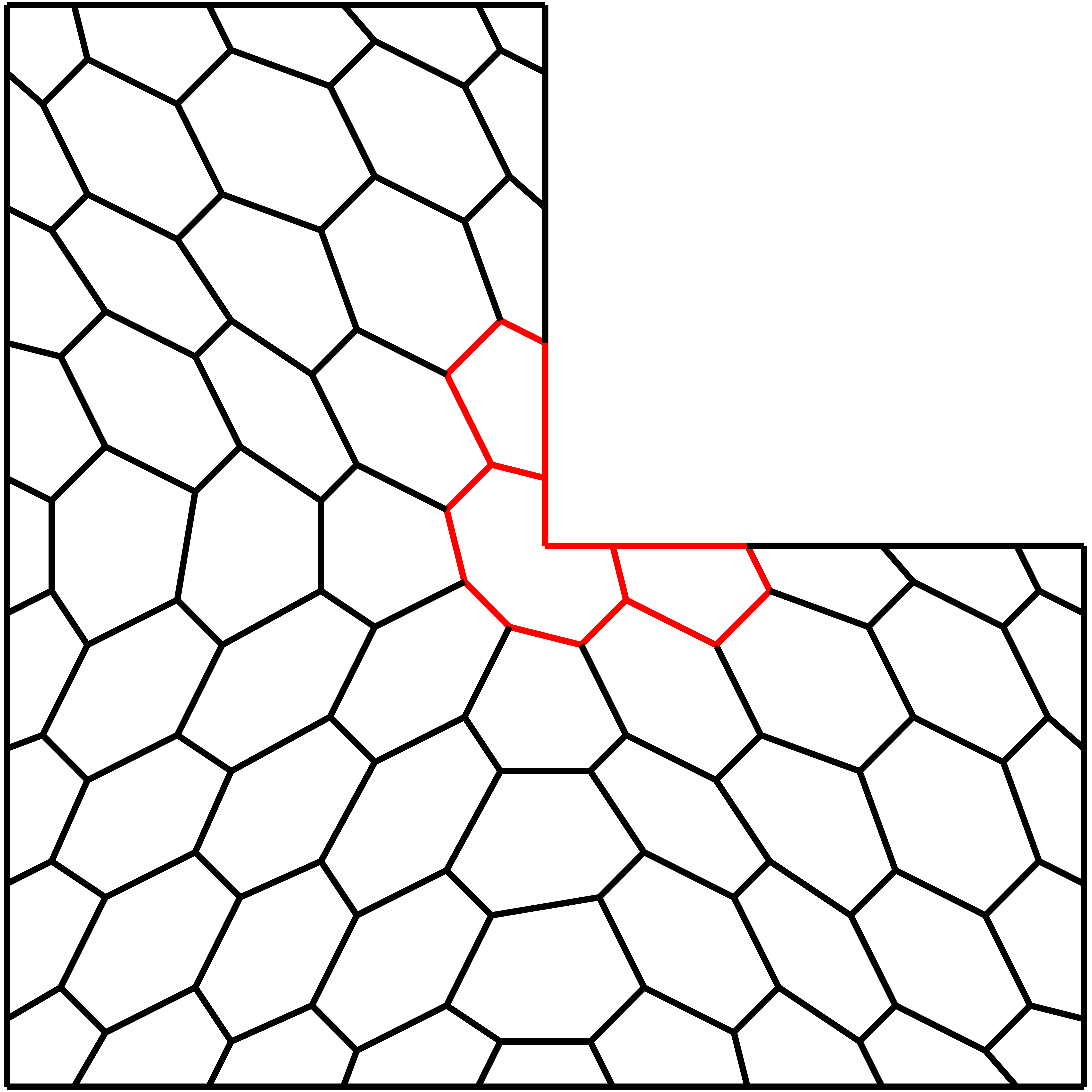}\hspace{0.5cm}
	\includegraphics[width=0.4\textwidth]{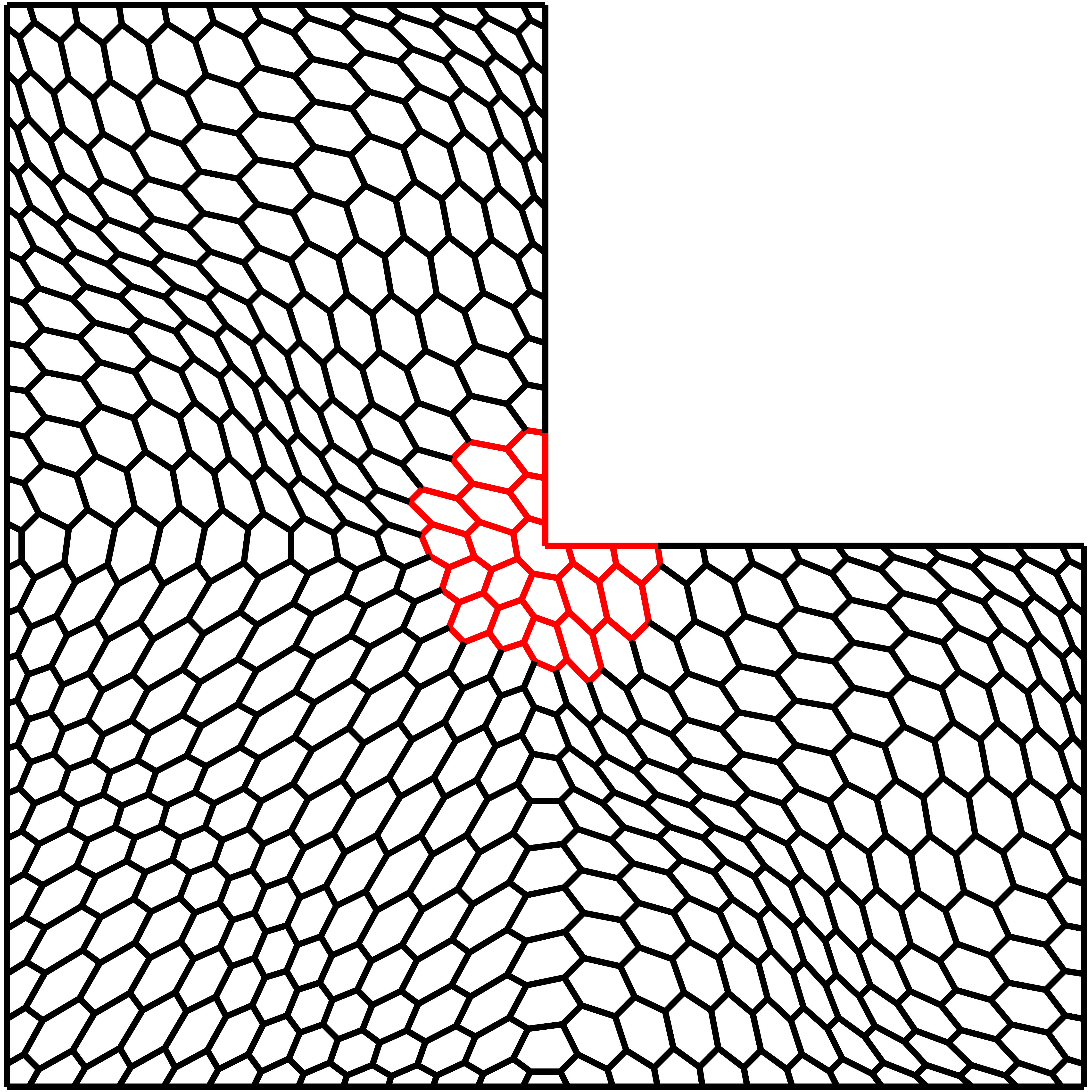} 
	\caption{Plots of meshes 2 and 4 used in the L-shaped domain test of Section \ref{sec:Lshape.1}. The enriched elements are coloured red.}
	\label{fig:hexa.meshes}
\end{figure}

Figure \ref{fig:lshape} illustrates the error of the method plotted against the remaining internal DOFs after static condensation for a non-enriched, globally enriched, and locally enriched scheme. As with the fractured domain, the globally enriched approach exhibits evident ill-conditioning, as observed by its failure on Mesh 7 when $k=1$, Mesh 6 \& 7 when $k=2$ and on Mesh 4, 5, 6 \& 7 when $k=3$. Indeed, for $k=1,2$, the error is greater for the globally enriched scheme than for either the non-enriched or locally enriched schemes. As with the fractured domain test, the locally enriched scheme is the best performing and converges optimally in all cases.

\begin{figure}[!ht]
	\centering
	\ref{legend_name}
	\vspace{0.25cm}\\
	\subcaptionbox{$E_{0,h}$ vs DOFs, $k=1$}
	{
		\begin{tikzpicture}[scale=0.75]
			\begin{loglogaxis}
				\addplot table[x=DOFs,y=L2Error] {orthonormalised_Lshape_nonenrich_k1.dat};
				\addplot table[x=DOFs,y=L2Error] {orthonormalised_Lshape_enrich_k1.dat};
				\addplot table[x=DOFs,y=L2Error] {orthonormalised_Lshape_locally_enrich_015_k1.dat};
				\reverseLogLogSlopeTriangle{0.50}{0.4}{0.1}{1}{black};
			\end{loglogaxis}
		\end{tikzpicture}
	}\hspace{0.5cm}\subcaptionbox{$E_{1,h}$ vs DOFs, $k=1$}
	{
		\begin{tikzpicture}[scale=0.75]
			\begin{loglogaxis}
				\addplot table[x=DOFs,y=H1Error] {orthonormalised_Lshape_nonenrich_k1.dat};
				\addplot table[x=DOFs,y=H1Error] {orthonormalised_Lshape_enrich_k1.dat};
				\addplot table[x=DOFs,y=H1Error] {orthonormalised_Lshape_locally_enrich_015_k1.dat};
				\reverseLogLogSlopeTriangle{0.50}{0.4}{0.1}{0.5}{black};
			\end{loglogaxis}
		\end{tikzpicture}
	}
	\vspace{0.25cm}\\
	\subcaptionbox{$E_{0,h}$ vs DOFs, $k=2$}
	{
		\begin{tikzpicture}[scale=0.75]
			\begin{loglogaxis}
				\addplot table[x=DOFs,y=L2Error] {orthonormalised_Lshape_nonenrich_k2.dat};
				\addplot table[x=DOFs,y=L2Error] {orthonormalised_Lshape_enrich_k2.dat};
				\addplot table[x=DOFs,y=L2Error] {orthonormalised_Lshape_locally_enrich_015_k2.dat};
				\reverseLogLogSlopeTriangle{0.50}{0.4}{0.1}{1}{black};
			\end{loglogaxis}
		\end{tikzpicture}
	}
	\hspace{0.5cm}\subcaptionbox{$E_{1,h}$ vs DOFs, $k=2$}
	{
		\begin{tikzpicture}[scale=0.75]
			\begin{loglogaxis}
				\addplot table[x=DOFs,y=H1Error] {orthonormalised_Lshape_nonenrich_k2.dat};
				\addplot table[x=DOFs,y=H1Error] {orthonormalised_Lshape_enrich_k2.dat};
				\addplot table[x=DOFs,y=H1Error] {orthonormalised_Lshape_locally_enrich_015_k2.dat};
				\reverseLogLogSlopeTriangle{0.50}{0.4}{0.1}{1}{black};
			\end{loglogaxis}
		\end{tikzpicture}
	}
	\vspace{0.25cm}\\
	\subcaptionbox{$E_{0,h}$ vs DOFs, $k=3$}
	{
		\begin{tikzpicture}[scale=0.75]
			\begin{loglogaxis}
				\addplot table[x=DOFs,y=L2Error] {orthonormalised_Lshape_nonenrich_k3.dat};
				\addplot table[x=DOFs,y=L2Error] {orthonormalised_Lshape_enrich_k3.dat};
				\addplot table[x=DOFs,y=L2Error] {orthonormalised_Lshape_locally_enrich_015_k3.dat};
				\reverseLogLogSlopeTriangle{0.50}{0.4}{0.1}{2}{black};
			\end{loglogaxis}
		\end{tikzpicture}
	}
	\hspace{0.5cm}\subcaptionbox{$E_{1,h}$ vs DOFs, $k=3$}
	{
		\begin{tikzpicture}[scale=0.75]
			\begin{loglogaxis}
				\addplot table[x=DOFs,y=H1Error] {orthonormalised_Lshape_nonenrich_k3.dat};
				\addplot table[x=DOFs,y=H1Error] {orthonormalised_Lshape_enrich_k3.dat};
				\addplot table[x=DOFs,y=H1Error] {orthonormalised_Lshape_locally_enrich_015_k3.dat};
				\reverseLogLogSlopeTriangle{0.50}{0.4}{0.1}{1.5}{black};
			\end{loglogaxis}
		\end{tikzpicture}
	}
	\caption{Tests on L-shaped domain, hexagonal meshes}
	\label{fig:lshape}
\end{figure}
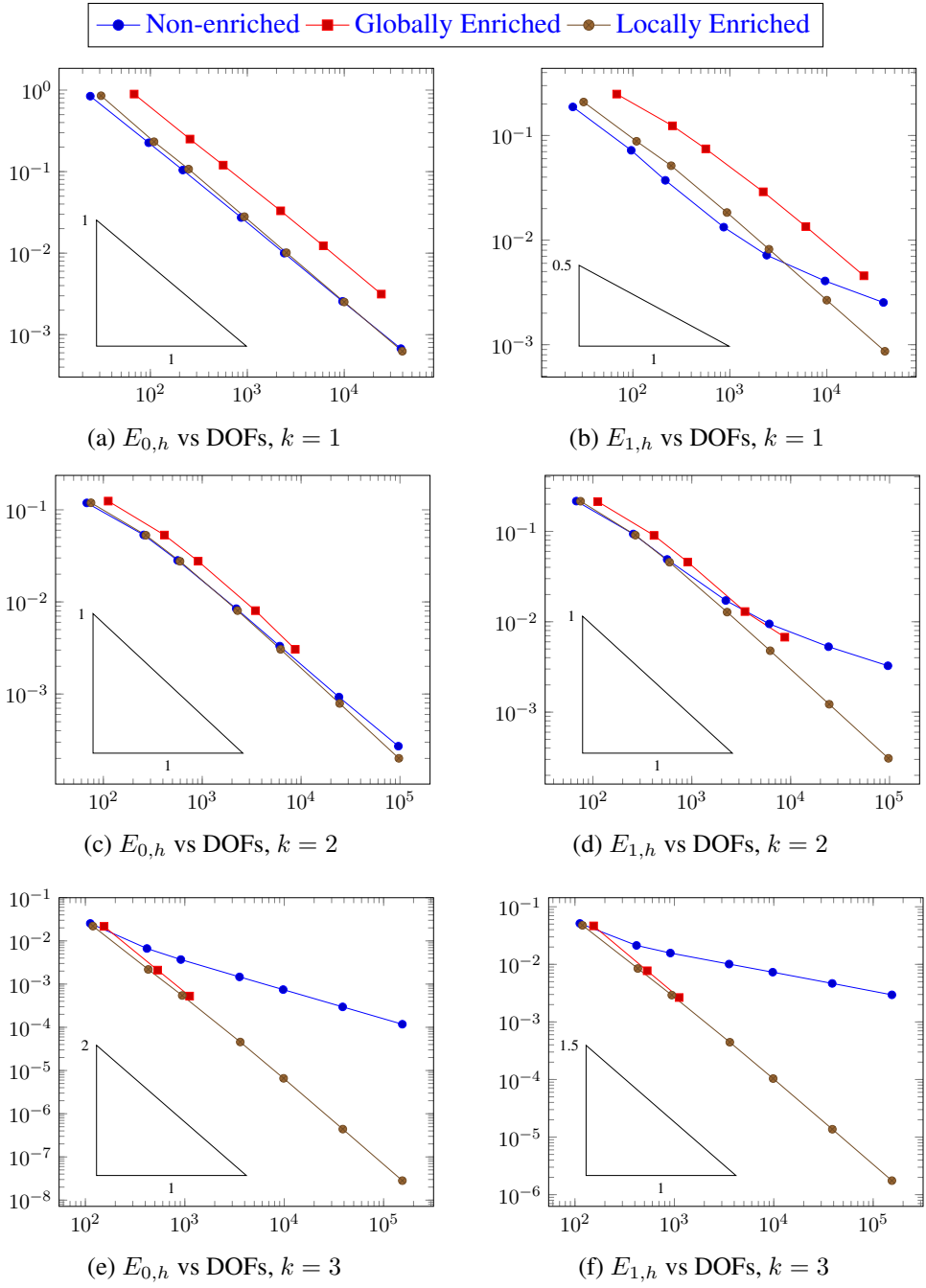

\subsection{L-shaped domain, Cartesian meshes}

In this section, we consider the test case of \cite[Section 5.1.1]{artioli.mascotto:2021:enrichment}. The domain is still an L-shaped one, albeit with the bottom right corner removed: $\Omega=(-1, 1)^2\backslash [0,1) \times (-1, 0]$. The exact solution and singular function are
\[
	u = \sin(\pi x) \sin (\pi y) + \psi\quad\text{ and }\quad\psi(r, \theta) = r^\frac23 \sin(\frac23 \theta).
\]
We run the simulations using uniform Cartesian meshes (see their characteristics in Table \ref{table:cartesian.Lshape.data}).
The convergence graphs for the globally, locally and non-enriched schemes are presented in Figure \ref{fig:lshape-cart}; note that there is no data for the non-enriched scheme with $k=1$ on the coarsest mesh since this mesh only has boundary vertices and, as a consequence, there are zero DOFs for that scheme on this mesh.

As in the previous tests, the locally enriched version outperforms the other two, with the non-enriched version displaying a sub-optimal convergence rate while the globally enriched version shows optimal convergence but larger magnitudes of error. The initial error growths in Figure \ref{fig:lshape-cart} are probably due to the use in the denominator of $E_{1,h}$ of the mesh-dependent quantity $(\sum_\P\snorm{\HS{1}}{\PinPt{k}\us}^2)^{1/2}$, which fluctuates a lot on the first coarse meshes before stabilising to a number close to $\snorm{\HS{1}(\Omega)}{\us}$ on finer meshes.
The ill-conditioning issues of \XVEM{} mentioned above are visible in Figure \ref{fig:lshape-cart} in the sense that the linear solver fails to converge on the finest mesh for the globally enriched scheme with degrees $k=2,3$; here too, we notice that local enrichment leads to a better conditioned scheme.

These results are comparable to the ones obtained with the enriched NCVEM in \cite[Fig. 11]{artioli.mascotto:2021:enrichment}. A direct quantitative comparison is however not feasible, because the errors in VEM and NCVEM are measured in different norms (through the use of scheme-specific reconstruction operators). The ill-conditioning of the globally enriched NCVEM is visible in \cite[Fig. 11]{artioli.mascotto:2021:enrichment} as a blow-up of the error. A mitigation was proposed there using new degrees of freedom corresponding to orthonormal polynomial basis functions on the edges. We also use such edge basis functions, but there is a fundamental difference between conforming VEM and NCVEM (and HHO), which could explain the worse ill-conditioning encountered in the former method: in NCVEM/HHO, the edge degrees of freedom directly give the boundary contributions in the local equations solved to compute the elliptic projectors/potential reconstructions (see \cite[Eqs.~(68) and following]{artioli.mascotto:2021:enrichment} or \cite[Eq.~(2.8)]{yemm:2022:design}), so the orthonormalisation of the bases corresponding to these degrees of freedom has a direct beneficial aspect in the computation of this boundary contribution; on the contrary, in VEM, the boundary contributions to the elliptic projectors involves the whole boundary value of the virtual function (see \eqref{eq:PinPt:computability}), which do not directly correspond to edge degrees of freedom but has to be reconstructed from vertex and edge degrees of freedom; this reconstruction probably reduces the positive impact of having chosen orthonormal bases for polynomial edge spaces.

\begin{table}[!ht]
\centering
\pgfplotstableread{orthonormalised_Lshape_cartesian_nonenriched_k1.dat}\loadedtable
\pgfplotstabletypeset
[
columns={MeshSize,NbCells,NbEdges,NbVertices}, 
columns/MeshSize/.style={column name=\(h\),/pgf/number format/.cd,fixed,zerofill,precision=4},
columns/NbCells/.style={column name=Nb. Elements},
columns/NbEdges/.style={column name=Nb. Edges},
columns/NbVertices/.style={column name=Nb. Vertices},
every head row/.style={before row=\toprule,after row=\midrule},
every last row/.style={after row=\bottomrule} 
]\loadedtable
\caption{Uniform Cartesian meshes for the L-shaped domain}
\label{table:cartesian.Lshape.data}
\end{table}

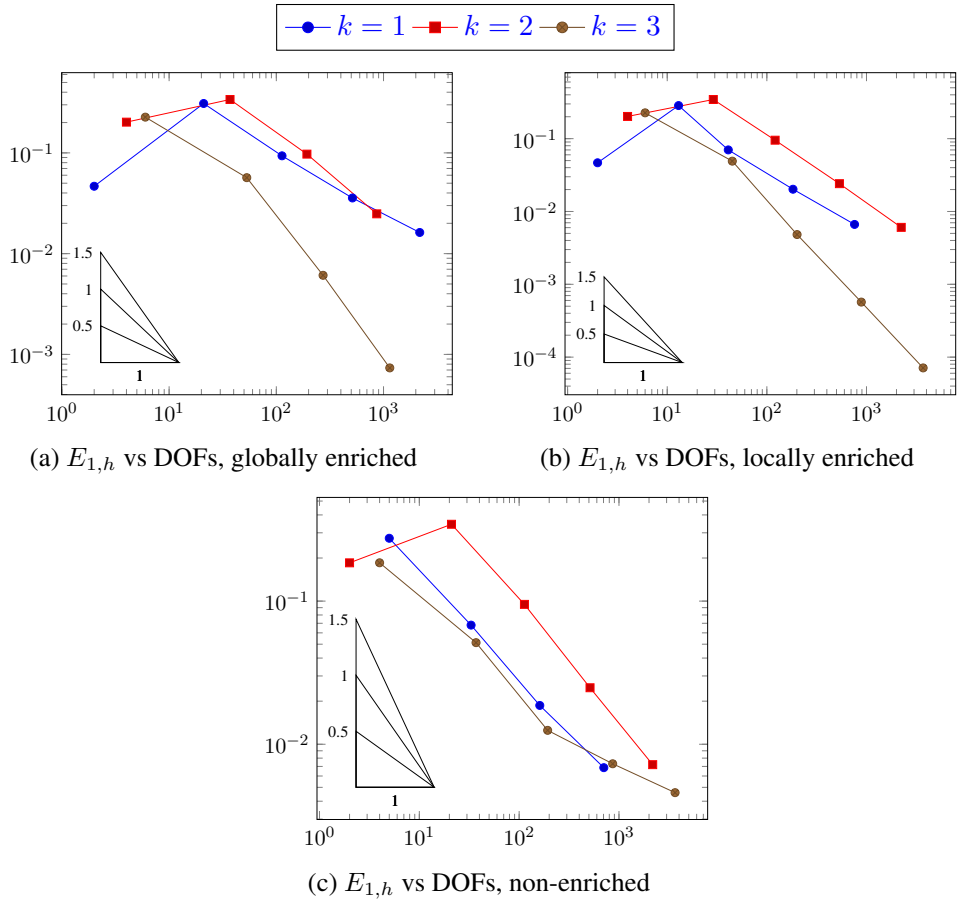
\begin{figure}[!ht]
\centering
\ref{cartesian_legend}
\vspace{0.25cm}\\
\subcaptionbox{$E_{1,h}$ vs DOFs, globally enriched}
{
	\begin{tikzpicture}[scale=0.75]
		\begin{loglogaxis}[ legend columns=3, legend to name=cartesian_legend ]
			\legend{$k=1$,$k=2$,$k=3$};
			\addplot table[x=DOFs,y=H1Error] {orthonormalised_Lshape_cartesian_enriched_k1.dat};
			\addplot table[x=DOFs,y=H1Error] {orthonormalised_Lshape_cartesian_enriched_k2.dat};
			\addplot table[x=DOFs,y=H1Error] {orthonormalised_Lshape_cartesian_enriched_k3.dat};
			\reverseLogLogSlopeTriangle{.3}{0.2}{0.1}{0.5}{black};
			\reverseLogLogSlopeTriangle{.3}{0.2}{0.1}{1}{black};
			\reverseLogLogSlopeTriangle{.3}{0.2}{0.1}{1.5}{black};
		\end{loglogaxis}
	\end{tikzpicture}
}\hspace{0.5cm}\subcaptionbox{$E_{1,h}$ vs DOFs, locally enriched}
{
	\begin{tikzpicture}[scale=0.75]
		\begin{loglogaxis}
			\addplot table[x=DOFs,y=H1Error] {orthonormalised_Lshape_cartesian_enriched_015_k1.dat};
			\addplot table[x=DOFs,y=H1Error] {orthonormalised_Lshape_cartesian_enriched_015_k2.dat};
			\addplot table[x=DOFs,y=H1Error] {orthonormalised_Lshape_cartesian_enriched_015_k3.dat};
			\reverseLogLogSlopeTriangle{.3}{0.2}{0.1}{0.5}{black};
			\reverseLogLogSlopeTriangle{.3}{0.2}{0.1}{1}{black};
			\reverseLogLogSlopeTriangle{.3}{0.2}{0.1}{1.5}{black};
		\end{loglogaxis}
	\end{tikzpicture}
}
\vspace{0.25cm}\\
\subcaptionbox{$E_{1,h}$ vs DOFs, non-enriched}
{
	\begin{tikzpicture}[scale=0.75]
		\begin{loglogaxis}
			\addplot table[x=DOFs,y=H1Error] {orthonormalised_Lshape_cartesian_nonenriched_k1.dat};
			\addplot table[x=DOFs,y=H1Error] {orthonormalised_Lshape_cartesian_nonenriched_k2.dat};
			\addplot table[x=DOFs,y=H1Error] {orthonormalised_Lshape_cartesian_nonenriched_k3.dat};
			\reverseLogLogSlopeTriangle{.3}{0.2}{0.1}{0.5}{black};
			\reverseLogLogSlopeTriangle{.3}{0.2}{0.1}{1}{black};
			\reverseLogLogSlopeTriangle{.3}{0.2}{0.1}{1.5}{black};
		\end{loglogaxis}
	\end{tikzpicture}
}
\caption{Tests on L-shaped domain, Cartesian meshes}
\label{fig:lshape-cart}
\end{figure}


\section{Conclusions}
\label{sec6:conclusions}

In this paper, a novel extension of the virtual element method has been introduced, which achieves consistency on highly generic enrichment spaces. This is applicable to problems posed in complex geometries, where singularities are known to exist near corners and fractures. However, the method also has applicability to any scenario where some component of the solution can be asymptotically obtained. This could include, for example, singular behaviour in the source term, or a highly oscillatory component of the solution. While the method was formulated for the Poisson problem in two dimensions, it has natural extensions to higher dimensions and more general linear elliptic problems.

We performed a complete analysis of the proposed method, proving its capacity to achieve optimal convergence rate in discrete and continuous energy norm. Numerical tests support the theoretical findings in the presence of corner and fracture singularities in both $L^2$- and $H^1$-norm. However, ill-conditioning issues plagues the globally enriched scheme -- something also observed by \cite{yemm:2022:design, artioli.mascotto:2021:enrichment}. The analysis in this paper does not account for the method of local enrichment considered in Section \ref{sec5:numerical}, however, numerical tests show optimal convergence. Providing a robust analysis for the \XVEM{} with discontinuous enrichment spaces is a potential avenue of future research.

\section*{Acknowledgement}
The work of JD was partially supported by the Australian Government
through the Australian Research Council's Discovery Projects funding
scheme (project number DP210103092) and by the European Union (ERC
Synergy, NEMESIS, project number 101115663); views and opinions
expressed are however those of the authors only and do not necessarily
reflect those of the European Union or the European Research Council
Executive Agency. Neither the European Union nor the granting
authority can be held responsible for them.

The work of GM was partially supported by the Laboratory Directed
Research and Development (LDRD) program of Los Alamos National
Laboratory under project number 20220129ER. Los Alamos National
Laboratory is operated by Triad National Security, LLC, for the
National Nuclear Security Administration of U.S. Department of Energy
(Contract No.\ 89233218CNA000001).
GM is a member of the Gruppo Nazionale Calcolo 
Scientifico-Istituto Nazionale di Alta Matematica (GNCS-INdAM).

The work of LY was supported by a Monash University Postgraduate
Publications Award.

\printbibliography

@Book{Adams-Fournier:2003,
  author    = {Adams, R.~A. and Fournier, J.~J.~F.},
  title     = {Sobolev spaces},
  publisher = {Academic Press},
  year      = {2003},
  series    = {Pure and Applied Mathematics},
  edition   = {2},
}

@article{BDDC-VEM,
author={Bertoluzza, S. and Pennacchio, M. and Prada, D.},
title={BDDC and FETI-DP for the virtual element method},
journal={Calcolo},
volume={54},
year={2017},
pages={1565--1593},
}

@article{cond_hho,
  AUTHOR = {Badia, Santiago and Droniou, J\'{e}r\^{o}me and Yemm, Liam},
  TITLE = {Conditioning of a hybrid high-order scheme on meshes with small faces},
  JOURNAL = {J. Sci. Comput.},
  FJOURNAL = {Journal of Scientific Computing},
  VOLUME = {92},
  YEAR = {2022},
  NUMBER = {2},
  PAGES = {Paper No. 71, 23},
  ISSN = {0885-7474},
  MRCLASS = {65N30 (65N12 65N15)},
  MRNUMBER = {4453262},
  DOI = {10.1007/s10915-022-01913-9},
  URL={https://arxiv.org/abs/2109.09983},
}

@article{brenner.sung:2018:virtual,
	title={Virtual element methods on meshes with small edges or faces},
	author={Brenner, Susanne C. and Sung, Li-Yeng},
	journal={Mathematical Models and Methods in Applied Sciences},
	volume={28},
	number={07},
	pages={1291--1336},
	year={2018},
	publisher={World Scientific}
}

@article{artioli.mascotto:2021:enrichment,
  title={Enrichment of the nonconforming virtual element method with singular functions},
  author={Artioli, Edoardo and Mascotto, Lorenzo},
  journal={Computer Methods in Applied Mechanics and Engineering},
  volume={385},
  pages={114024},
  year={2021},
  publisher={Elsevier}
}

@article{belytschko.black:1999:elastic,
  title={Elastic crack growth in finite elements with minimal remeshing},
  author={Belytschko, Ted and Black, Tom},
  journal={International journal for numerical methods in engineering},
  volume={45},
  number={5},
  pages={601--620},
  year={1999},
  publisher={Wiley Online Library}
}

@Article{	  Cockburn.Gopalakrishnan.ea:09,
  author	= {Cockburn, B. and Gopalakrishnan, J. and Lazarov, R.},
  title		= {Unified hybridization of discontinuous {G}alerkin, mixed,
		  and continuous {G}alerkin methods for second order elliptic
		  problems},
  journal	= {SIAM J. Numer. Anal.},
  volume	= {47},
  year		= {2009},
  number	= {2},
  pages		= {1319--1365},
  doi		= {10.1137/070706616}
}

@Article{	  Cockburn.Di-Pietro.ea:16,
  author	= {Cockburn, B. and Di Pietro, D. A. and Ern, A.},
  title		= {Bridging the {Hybrid High-Order} and {Hybridizable
		  Discontinuous Galerkin} methods},
  journal	= {ESAIM: Math. Model. Numer. Anal.},
  year		= {2016},
  volume	= {50},
  number	= {3},
  pages		= {635--650},
  doi		= {10.1051/m2an/2015051}
}

@Article{	  Cangiani.Dong.ea:16,
  author	= {Cangiani, A. and Dong, Z. and Georgoulis, E. H. and
		  Houston, P.},
  title		= {{$hp$}-version discontinuous {G}alerkin methods for
		  advection-diffusion-reaction problems on polytopic meshes},
  journal	= {ESAIM Math. Model. Numer. Anal.},
  volume	= {50},
  year		= {2016},
  number	= {3},
  pages		= {699--725},
  doi		= {10.1051/m2an/2015059}
}

@Article{	  Antonietti.Brezzi.ea:09,
  author	= {Antonietti, Paola F. and Brezzi, Franco and Marini, L.
		  Donatella},
  title		= {Bubble stabilization of discontinuous {G}alerkin methods},
  journal	= {Comput. Methods Appl. Mech. Engrg.},
  volume	= {198},
  year		= {2009},
  number	= {21-26},
  pages		= {1651--1659},
  doi		= {10.1016/j.cma.2008.12.033}
}

@Book{		  Di-Pietro.Ern:12,
  author	= {Di Pietro, D. A. and Ern, A.},
  title		= {Mathematical aspects of discontinuous {G}alerkin methods},
  series	= {Math\'ematiques \& Applications (Berlin) [Mathematics \&
		  Applications]},
  volume	= {69},
  publisher	= {Springer, Heidelberg},
  year		= {2012},
  pages		= {xviii+384},
  isbn		= {978-3-642-22979-4},
  doi		= {10.1007/978-3-642-22980-0}
}

@article{moes.dolbow:1999:finite,
  title={A finite element method for crack growth without remeshing},
  author={Mo{\"e}s, Nicolas and Dolbow, John and Belytschko, Ted},
  journal={International journal for numerical methods in engineering},
  volume={46},
  number={1},
  pages={131--150},
  year={1999},
  publisher={Wiley Online Library}
}

@article{melenk.babuska:1996:partition,
  title={The partition of unity finite element method: basic theory and applications},
  author={Melenk, Jens M. and Babu\v{s}ka, Ivo},
  journal={Computer methods in applied mechanics and engineering},
  volume={139},
  number={1-4},
  pages={289--314},
  year={1996},
  publisher={Elsevier}
}

@article{babuska.melenk:1997:partition,
  title={The partition of unity method},
  author={Babu\v{s}ka, Ivo and Melenk, Jens M.},
  journal={International journal for numerical methods in engineering},
  volume={40},
  number={4},
  pages={727--758},
  year={1997},
  publisher={Wiley Online Library}
}

@article{benvenuti.chiozzi.ea:2022:extended,
  title={Extended virtual element method for two-dimensional linear elastic fracture},
  author={Benvenuti, Elena and Chiozzi, Andrea and Manzini, Gianmarco and Sukumar, Natarajan},
  journal={Computer Methods in Applied Mechanics and Engineering},
  volume={390},
  pages={114352},
  year={2022},
  publisher={Elsevier}
}

@article{artioli.mascotto:2023:enriched,
  title={Enriched virtual elements for plane elasticity with corner singularities},
  author={Artioli, Edoardo and Mascotto, Lorenzo},
  doi={10.1007/s00466-023-02418-4},
  journal={Comput. Mech.},
  year={2023}
}

@article{yemm:2022:design,
  title={Design and analysis of the extended hybrid high-order method for the Poisson problem},
  author={Yemm, Liam},
  journal={Advances in Computational Mathematics},
  volume={48},
  number={4},
  pages={1--25},
  year={2022},
  publisher={Springer},
  doi={10.1007/s10444-022-09958-y}
}

@article{benvenuti.chiozzi.ea:2019:extended,
  title={Extended virtual element method for the Laplace problem with singularities and discontinuities},
  author={Benvenuti, Elena and Chiozzi, Andrea and Manzini, Gianmarco and Sukumar, Natarajan},
  journal={Computer Methods in Applied Mechanics and Engineering},
  volume={356},
  pages={571--597},
  year={2019},
  publisher={Elsevier}
}

@online{PolyMesh,
author = {Yemm, Liam},
title = {PolyMesh},
date = {2024},
url = {https://github.com/liamyemm/polymesh}
}

@article{yemm:2024:enriched,
  title={An enriched hybrid high-order method for the Stokes problem with application to flow around submerged cylinders},
  author={Yemm, Liam},
 url={https://arxiv.org/abs/2311.13243},
  year={2024}
}

@article{Ahmad-Alsaedi-Brezzi-Marini-Russo:2013,
  author    = {Ahmad, B. and Alsaedi, A. and Brezzi, F. and Marini,
                  L.~D. and Russo, A.},
  journal   = {Computers \& Mathematics with Applications},
  month     = {September},
  pages     = {376--391},
  title     = {Equivalent projectors for virtual element methods},
  volume    = {66},
  year      = {2013},
}

@article{BeiraodaVeiga-Brezzi-Cangiani-Manzini-Marini-Russo:2013,
  author    = {Beir\~ao da Veiga, L. and Brezzi, F. and Cangiani,
                  A. and Manzini, G. and Marini, L. D. and Russo, A.},
  journal   = {Mathematical Models \& Methods in Applied Sciences},
  pages     = {119--214},
  title     = {Basic principles of virtual element methods},
  volume    = {23},
  year      = {2013},
}

@article{Brenner-Guan-Sung:2017:SEV,
  author    = {Brenner, S.~C. and Guan, Q. and Sung, L.-Y.},
  journal   = {Computational Methods in Applied Mathematics},
  number    = {4},
  pages     = {553--574},
  title     = {Some estimates for virtual element methods},
  volume    = {17},
  year      = {2017},
}

@Book{		  Beirao-da-Veiga.Lipnikov.ea:14,
  author	= {Beir\~ao da Veiga, L. and Lipnikov, K. and Manzini, G.},
  title		= {The mimetic finite difference method for elliptic
		  problems},
  series	= {MS\&A. Modeling, Simulation and Applications},
  volume	= {11},
  publisher	= {Springer, Cham},
  year		= {2014},
  pages		= {xvi+392},
  isbn		= {978-3-319-02662-6; 978-3-319-02663-3},
  doi		= {10.1007/978-3-319-02663-3}
}

@Article{	  Kuznetsov.Lipnikov.ea:04,
  author	= {Kuznetsov, Y. and Lipnikov, K. and Shashkov, M.},
  title		= {Mimetic finite difference method on polygonal meshes for
		  diffusion-type problems},
  journal	= {Comput. Geosci.},
  volume	= {8},
  year		= {2004},
  pages		= {301--324}
}

@Article{	  Eymard.Gallouet.ea:10,
  author	= {Eymard, R. and Gallou\"{e}t, T. and Herbin, R.},
  title		= {Discretization of heterogeneous and anisotropic diffusion
		  problems on general nonconforming meshes. {SUSHI}: a scheme
		  using stabilization and hybrid interfaces},
  journal	= {IMA J. Numer. Anal.},
  year		= {2010},
  volume	= {30},
  number	= {4},
  pages		= {1009--1043},
  doi		= {10.1093/imanum/drn084}
}

@Article{	  Droniou.Eymard.ea:10,
  title		= {A unified approach to mimetic finite difference, hybrid
		  finite volume and mixed finite volume methods},
  author	= {Droniou, J. and Eymard, R. and Gallou\"{e}t, T. and
		  Herbin, R.},
  journal	= {Math. Models Methods Appl. Sci. (M3AS)},
  volume	= {20},
  number	= {2},
  pages		= {1--31},
  year		= {2010},
  doi		= {10.1142/S0218202510004222}
}

@Article{	  Droniou:14,
  author	= {Droniou, J.},
  title		= {Finite volume schemes for diffusion equations:
		  introduction to and review of modern methods},
  journal	= {Math. Models Methods Appl. Sci.},
  volume	= {24},
  year		= {2014},
  number	= {8},
  pages		= {1575--1619},
  doi		= {10.1142/S0218202514400041}
}

@Article{	  Mu.Wang.ea:15,
  author	= {Mu, Lin and Wang, Junping and Ye, Xiu},
  title		= {Weak {G}alerkin finite element methods on polytopal
		  meshes},
  journal	= {Int. J. Numer. Anal. Model.},
  volume	= {12},
  year		= {2015},
  number	= {1},
  pages		= {31--53}
}

@Article{	  Di-Pietro.Ern.ea:14,
  author	= {Di Pietro, D. A. and Ern, A. and Lemaire, S.},
  title		= {An arbitrary-order and compact-stencil discretization of
		  diffusion on general meshes based on local reconstruction
		  operators},
  journal	= {Comput. Meth. Appl. Math.},
  volume	= {14},
  number	= {4},
  pages		= {461--472},
  year		= {2014},
  doi		= {10.1515/cmam-2014-0018}
}

@Article{         Di-Pietro.Droniou:23,
  author        = {Di Pietro, Daniele A. and Droniou, J\'er\^ome},
  title         = {An arbitrary-order discrete de Rham complex on polyhedral
                  meshes: Exactness, Poincar\'e inequalities, and
                  consistency},
  journal       = {Found. Comput. Math.},
  volume        = {23},
  pages         = {85--164},
  year          = {2023},
  doi           = {10.1007/s10208-021-09542-8}
}

@book{hho-book,
	TITLE={The Hybrid High-Order Method for Polytopal Meshes: Design, Analysis, and Applications},
	AUTHOR={Di Pietro, Daniele Antonio and Droniou, J\'er\^ome},
	YEAR={2020},
	PAGES={xxxi + 525p},
	SERIES = {Modeling, Simulation and Applications},
	PUBLISHER = {Springer International Publishing},
  volume        = {19},
  isbn          = {978-3-030-37202-6 (Hardcover), 978-3-030-37203-3 (eBook)},
  doi           = {10.1007/978-3-030-37203-3},
	URL={https://hal.archives-ouvertes.fr/hal-02151813},
}

@book{Grisvard:1985,
  address   = {Boston, MA},
  author    = {Grisvard, P.},
  publisher = {Pitman Publishing, Inc},
  title     = {Elliptic Problems in Nonsmooth Domains},
  year      = {1985},
}

@Article{Motz:1947,
  author    = {Motz, M.},
  title     = {The treatment of singularities of partial differential equations by relaxation methods},
  journal   = {Quarterly of Applied Mathematics},
  year      = {1947},
  volume    = {4},
  pages     = {371--377},
}


\end{document}